\pdfoutput=1 
\documentclass[draft=false, leqno]{article}
\usepackage{amsmath,amsthm,mathrsfs}
\usepackage[normal]{boilerart1}
\usepackage{longtable}

\usepackage[style=british]{csquotes}
\newcommand{\defn}[1]{\emph{#1}}

\newcommand{\respectively}[1]{\text{\qquad(respectively,\quad}{#1}\text{\,)}}

\newcommand{\andeq}{\text{\qquad and\qquad}}
\newcommand{\period}{\rlap{\ .}}
\newcommand{\comma}{\rlap{\ ,}}

\newcommand{\Cech}{\check{C}}

\newcommand{\ilowerstar}{i_{\ast}}
\newcommand{\iupperstar}{i^{\ast}}

\newcommand{\flowerstar}{f_{\ast}}
\newcommand{\fupperstar}{f^{\ast}}

\newcommand{\xlowerstar}{x_{\ast}}

\newcommand{\Gammalowerstar}{\Gamma_{\ast}}
\newcommand{\Gammaupperstar}{\Gamma^{\ast}}

\newcommand{\Gammauppershriek}{\Gamma^{!}}

\newcommand{\iuppershriek}{i^{!}}

\DeclareMathOperator{\Env}{Env}
\DeclareMathOperator{\rank}{rk}

\newcommand{\fp}{\,\ensuremath{\textit{fp}}}
\renewcommand{\fin}{\,\ensuremath{\textit{fin}}}
\newcommand{\acc}{\ensuremath{\textit{acc}}}

\newcommand{\constr}{\ensuremath{\textit{constr}}}

\newcommand{\lc}{\ensuremath{\textit{lc}}}

\newcommand{\proet}{\ensuremath{\textit{proét}}}
\newcommand{\perf}{\ensuremath{\textit{perf}}}

\newcommand{\hyp}{\ensuremath{\textit{hyp}}}

\newcommand{\spec}{\ensuremath{\textit{spec}}}
\newcommand{\disc}{\ensuremath{\textit{disc}}}
\newcommand{\indisc}{\ensuremath{\textit{indisc}}}

\newcommand{\coh}{\ensuremath{\textit{coh}}}
\newcommand{\bc}{\ensuremath{\textit{bc}}}

\newcommand{\cts}{\ensuremath{\textit{cts}}}

\newcommand{\pt}{\ast}
\newcommand{\stable}{\ensuremath{\textit{st}}}

\newcommand{\lult}{\ensuremath{\textit{LUlt}}}
\newcommand{\cc}{\ensuremath{\textit{cc}}}
\newcommand{\pyk}{\ensuremath{\textit{pyk}}}

\newcommand{\cross}{\times}
\newcommand{\tensor}{\otimes}

\newcommand{\isomorphic}{\cong}

\newcommand{\equivalent}{\simeq}
\newcommand{\paren}[1]{\left(#1\right)}

\def\reflectedop#1#2{\mathop{\reflectbox{$#1{#2}$}}}

\newcommand{\Setfin}{\Set^{\fin}}

\newcommand{\Bool}{\categ{Bool}}
\newcommand{\Comp}{\categ{Comp}}

\newcommand{\Pyk}{\categ{Pyk}}

\newcommand{\TSpc}{\categ{TSpc}}

\newcommand{\Sh}{\categ{Sh}}
\newcommand{\Sheff}{\Sh_{\eff}}

\newcommand{\Topbc}{\Top_{\infty}^{\bc}}

\newcommand{\Spacefin}{\Space_{\pi}}

\newcommand{\Strat}{\categ{Str}}

\newcommand{\CG}{\categ{CG}}
\newcommand{\Boolcomp}{\categ{Bool}^{\wedge}}
\newcommand{\Stn}{\categ{Stn}}
\newcommand{\EStn}{\categ{EStn}}
\newcommand{\Stonean}{\categ{EStn}}
\newcommand{\CS}{\categ{CS}}
\newcommand{\Catfp}{\Cat^{\fp}}
\newcommand{\LCA}{\categ{LCA}}
\newcommand{\USet}{\categ{USet}}
\newcommand{\Ult}{\categ{Ult}}
\newcommand{\UltL}{\Ult^{L}}
\newcommand{\PykSp}{\Pyk^{\Sp}}
\newcommand{\Unif}{\categ{Unif}}
\newcommand{\PsiPyk}{{\upPsi\Pyk}}

\DeclareMathOperator{\Gal}{Gal}

\DeclareMathOperator{\mat}{mat}

\DeclareMathOperator{\Pro}{Pro}
\DeclareMathOperator{\rk}{rk}

\DeclareMathOperator{\Tot}{Tot}
\DeclareMathOperator{\Pt}{Pt}

\newcommand{\MOR}{\ensuremath{\textup{\textsc{Mor}}}}

\newcommand{\Funcross}{\Fun^{\cross}}
\newcommand{\Funlult}{\Fun^{\lult}}

\newcommand{\twarr}{\widetilde{O}}

\newcommand{\XXcoh}{\XX^{\coh}}

\newcommand{\XXcohbdd}{\XXcoh_{<\infty}}

\newcommand{\enumref}[2]{(\ref{#1}.\ref{#1.#2})} 

\numberwithin{equation}{subsection}

\usepackage{todonotes}
\usepackage{xargs}  
\newcommandx{\PHtodo}[2][1=]{\todo[linecolor=blue,backgroundcolor=blue!25,bordercolor=blue,#1]{#2}}

\title{Pyknotic objects, I. Basic notions}

\author{Clark Barwick \and Peter Haine}

\date{}

\begin{document}

\maketitle

\begin{abstract} 
	Pyknotic objects are (hyper)sheaves on the site of compacta.
	These provide a convenient way to do algebra and homotopy theory with additional topological information present.
	This appears, for example, when trying to contemplate the derived category of a local field.
	In this article, we present the basic theory of pyknotic objects, with a view to describing a simple set of everyday examples. 
\end{abstract}

\setcounter{tocdepth}{2}
\tableofcontents

\setcounter{section}{-1}


\section{Introduction}


\subsection{The proétale topology and pyknotic objects}

Let $E$ be a local field, and let $X$ be a connected, topologically noetherian, coherent scheme.
Bhargav Bhatt and Peter Scholze \cite[Lemma 7.4.7]{BhattScholzeProEtale} construct a topological group \smash{$\pi_1^{\,\proet}(X)$} that classifies local systems of $E$-vector spaces in the sense that there is a monodromy equivalence of categories between the continuous $E$-linear representations and $E$-linear local systems.
The group \smash{$\pi_1^{\,\proet}(X)$} isn't profinite or even a proöbject in discrete groups in general: Deligne's example of a curve of genus $\geq 1$ with two points identified has local systems that are not classified by any such group.

In forthcoming work \cite{exodromy}, we will extend the Bhatt--Scholze monodromy equivalence to an \emph{exodromy equivalence} between continuous $E$-representations of the Galois category $\Gal(X)$ and constructible sheaves of $E$-vector spaces.
To speak of such continuous representations, one needs to contemplate not only the category of finite dimensional $E$-vector spaces but also the natural topology thereupon.

To describe constructible sheaves of \emph{complexes} of $E$-vector spaces, we need a new idea in order to speak of an $\infty$-category of perfect complexes of $E$-vector spaces in a manner that retains the natural topological information coming from $E$.

In this paper, we describe a way to do this: a \emph{pyknotic\footnote{\textit{Pykno} comes from the Greek \textit{\textgreek{πυκνός}} meaning `dense', `compact', or `thick'.} object of an $\infty$-category $ C $} is a (hyper)sheaf on the site of compact hausdorff spaces valued in $ C $.
We may thus speak of pyknotic sets, pyknotic groups, pyknotic rings, pyknotic spaces, pyknotic $\infty$-categories, \textit{\& c.}
Pyknotic structures function in much the same way as topological structures.

At the same time, pyknotic sets are the proétale sheaves of sets on a separably closed field, and the proétale topos of any coherent scheme has the natural structure of a pyknotic category.
There is a deep connection between the passage from objects to pyknotic objects and the passage from the étale topology to the proétale topology.

Our local field $E$ is naturally a pyknotic ring; pyknotic vector spaces over $E$ comprise a pyknotic category; complexes of pyknotic vector spaces over $E$ comprise a pyknotic $\infty$-category $\DD(E)$; and perfect complexes of pyknotic vector spaces over $E$ comprise a pyknotic subcategory $\DD^{\perf}(E)$.
Our exodromy equivalence will then be a natural equivalence
\[
	\Fun^{\Pyk}(\Gal(X), \DD^{\perf}(E)) \simeq \DD^{\constr}_{\proet}(X; E) \period
\]

Moreover, the proétale $\infty$-topos $X_{\proet}$ itself is naturally a pyknotic category, and one can identify it with the category of pyknotic functors from $\Gal(X)$ to pyknotic spaces:
\[
	X_{\proet} \simeq \Fun^{\Pyk}(\Gal(X), \Pyk(\Space)) \period
\]


\subsection{The aims of this paper}

This paper is the first of a series.
Our objective here is only to establish the very basic formalism of pyknotic structures, in the interest of developing a few key examples.

\begin{exm}
	For any set, group, abelian group, ring, space, spectrum, category, \textit{\& c.}, $A$, there are both a discrete pyknotic object $A^{\disc}$ and an indiscrete pyknotic object $A^{\indisc}$ attached to $A$ (\Cref{cnstr:discindisc}).
	As with topological structures, these notions are set up so that a map out of a discrete object is determined by a map at the level of the underlying object, and a map into an indiscrete object is determined by a map at the level of the underlying object.
\end{exm}

\begin{exm}
	Starting with discrete objects, one can develop more interesting pyknotic structures by the formation of inverse limits.
	Thus profinite groups like Galois groups and étale fundamental groups are naturally pyknotic, and profinite categories like $\Gal(X)$ above are naturally pyknotic (\Cref{exm:exodromyaspyknotic}).
	These inverse limits are no longer discrete.
\end{exm}

\begin{exm}
	More generally still, compactly generated topological spaces embed fully faithfully into pyknotic sets, in a manner that preserves limits (\Cref{exm:CGembeds}).
	Thus locally compact abelian groups, normed rings, and complete locally convex topological vector spaces are all naturally pyknotic objects.
	This includes the vast majority of topological objects that appear in number theory and functional analysis.

	One key point, however, is that the relationship between compactly generated topological spaces and pyknotic sets is \emph{dual} to the relationship between compactly generated topological spaces and general topological spaces:
	in topological spaces, compactly generated topological spaces are stable under colimits but not limits;
	in pyknotic sets, compactly generated topological spaces are stable under limits but not colimits.

	Furthermore, since pyknotic sets form a $1$-topos, it follows readily that products of quotients are again quotients (\Cref{exm:quotientsarequotients}).
	This is of course not true in the realm of topological spaces, and this is one of the main reasons that topologising fundamental groups is such a fraught endeavour.
\end{exm}

\begin{exm}
	More exotically, the cokernel $\widehat{\ZZ}/\ZZ$ in pyknotic groups is not indiscrete. 
	This is in contrast with the topological case.

	Even more dramatically, if $A$ is a locally compact abelian group, the continuous homomorphism $i \colon A^{\disc} \to A$, when viewed as a pyknotic homomorphism, is a monomorphism with a nontrivial cokernel.
	The underlying abelian group of this cokernel, however, is trivial.
	This underscores one of the main peculiarities of the theory of pyknotic structures, which is also one of its advantages: the forgetful functor is not faithful.
\end{exm}

\begin{exm}
	Pyknotic spaces and spectra form well-behaved categories, and their homotopy groups are naturally pyknotic.
	This makes it sensible to speak of topologies on the homotopy groups of spaces and spectra.
	For example, the $E$-nilpotent completion of a spectrum is naturally a pyknotic spectrum (\Cref{exm:Enilpotentcomp}), and its homotopy pyknotic groups are computed by the $E$-based Adams--Novikov spectral sequence.
\end{exm}

\begin{exm}
	The category of pyknotic objects of a presentable category $C$ form a natural example of a \defn{pyknotic category}:
	the category of sections over any compactum $K$ is itself the category of sheaves in $C$ on the site of compacta over $K$.
	Pyknotic categories provide a context in which one can do homotopy theory while keeping control of `topological' structures.

	For example, for a local field $E$, one may speak of the \defn{pyknotic derived category} $\DD_{\Pyk}(E)$, whose objects can be thought of as complexes of pyknotic vector spaces over $E$.
	This construction will be the focus of our attention in a sequel to this paper.
\end{exm}


\subsection{Pyknotic and condensed}

As we were developing these ideas, we learned that Dustin Clausen and Peter Scholze have independently been studying essentially the same notion, which they call \emph{condensed objects}.\footnote{In fact, as we were preparing this first manuscript, Scholze's ongoing lecture notes \cite{Scholze:condensednotes} appeared and Scholze gave a talk at MSRI on this material \cite{Scholze:condensedtalknotes,Scholze:Condensedtalk}.}

There is, however, a difference between pyknotic objects and the condensed objects of Clausen and Scholze: it is a matter of set theory.
To explain this, select a strongly inaccessible cardinal $\delta$ and the smallest strongly inaccessible cardinal $\delta^+$ over $\delta$.
A pyknotic set in the universe $\VV_{\delta^+}$ is a sheaf on the site $\Comp_{\delta}$ of $\delta$-small compacta, valued in the category $\Set_{\delta^+}$ of $\delta^+$-small sets.
By contrast, a condensed set in the universe $\VV_{\delta}$ is a sheaf on $\Comp_{\delta}$ valued in $\Set_{\delta}$ that is in addition $\kappa$-accessible for some regular cardinal $\kappa<\delta$.
Thus condensed sets in $\VV_{\delta}$ embed fully faithfully into pyknotic sets in $\VV_{\delta^+}$, which in turn embed fully faithfully into condensed sets in $\VV_{\delta^+}$.
(We shall discuss this accessibility more precisely in \cref{app:A}.)

The Clausen--Scholze theory of condensed objects can thus be formalised completely in \textsc{zfc},
whereas our theory of pyknotic objects requires at least one strongly inaccessible cardinal.

As emphasised by Scholze, however, the distinction between pyknotic and condensed does have some consequences beyond philosophical matters.
For example, the indiscrete topological space $\{0,1\}$, viewed as a sheaf on the site of compacta, is pyknotic but not condensed (relative to any universe).
By allowing the presence of such pathological objects into the category of pyknotic sets, we guarantee that it is a topos, which is not true for the category of condensed sets.

It would be too glib to assert that the pyknotic approach values the niceness of the category over the niceness of its objects, while the condensed approach does the opposite.
However, it seems that the pyknotic objects that one will encounter in serious applications will usually be condensed, and the majority of the good properties of the category of condensed objects will usually be inherited from the category of pyknotic objects.


\subsection{Acknowledgements}

There is certainly overlap in our work here with that of Clausen and Scholze, even though our aims are somewhat different.
We emphasise that Clausen in particular had understood the significance of condensed objects for many years before we even started to contemplate them.
We thank both Clausen and Scholze for the insights (and corrections) they have generously shared with us via e-mail.

Even outside these private communications, our intellectual debt to them is, we hope, obvious.

We are also grateful to Jacob Lurie, who explained to us many ideas related to ultracategories, and in particular outlined for us the $\infty$-ultracategory material that will eventually be added to \cite{Kerodon}.


\section{Conventions}


\subsection{Higher categories}

\begin{nul}\label{nul:conventionshighercats}
	We use the language and tools of higher category theory, particularly in the model of \emph{quasicategories}, as defined by Michael Boardman and Rainer Vogt and developed by André Joyal and Jacob Lurie.
	We will generally follow the terminological and notational conventions of Lurie's trilogy \cites{HTT,HA,SAG}, but we will simplify matters by \emph{systematically using words to mean their good homotopical counterparts.}\footnote{We have grown weary of the practise of prefixing words with sequences of unsearchable crackjaw symbols.}
	\begin{itemize}
		\item
			The word \emph{category} here will always mean \emph{$\infty$-category} or \emph{$(\infty,1)$-category} or \emph{quasicategory} -- i.e., a simplicial set satisfying the weak Kan condition.
		\item
			A \emph{subcategory} $C'$ of a category $ C $ is a simplicial subset that is stable under composition in the strong sense, so that if $\sigma \colon \Delta^n \to C$ is an $n$-simplex of $ C $, then $\sigma$ factors through $C'\subseteq C$ if and only if each of the edges $\sigma(\Delta^{\{i,i+1\}})$ does so.
		\item
			We will use the terms \emph{groupoid} or \emph{space} interchangeably for what is often called an \emph{$\infty$-groupoid} -- i.e., a category in which every morphism is invertible.

		\item For a category $ C $, we write $ \Pro(C) $ for the category of \textit{proöbjects} in $ C $.
	\end{itemize}
\end{nul}


\subsection{Set theoretic conventions}

\begin{nul}
	Recall that if $\delta$ is a strongly inaccessible cardinal (which we always assume to be uncountable), then the set $\VV_{\delta}$ of all sets of rank strictly less than $\delta$ is a Grothendieck universe \cite[Exposé I, Appendix]{MR50:7130} of rank and cardinality $\delta$.
	Conversely, if $\VV$ is a Grothendieck universe that contains an infinite cardinal, then $\VV=\VV_{\delta}$ for some inaccessible cardinal $\delta$.

	In order to deal precisely and simply with set-theoretic problems arising from some of the `large' operations, we append to \textsc{zfc} the Axiom of Universes (\textsc{au}).
	This asserts that any cardinal is dominated by a strongly inaccessible cardinal.

	We write $\delta_0$ for the smallest strongly inaccessible cardinal.
	Now \textsc{au} implies the existence of a hierarchy of strongly inaccessible cardinals
	\[
		\delta_0 < \delta_1 < \delta_2 < \cdots \comma
	\]
	in which for each ordinal $\alpha$, the cardinal $\delta_{\alpha}$ is the smallest strongly inaccessible cardinal $\delta_{\alpha}$ that dominates $\delta_{\beta}$ for any $\beta<\alpha$.\footnote{Thus $\VV_{\delta_{\alpha}}$ models \textsc{zfc} plus the axiom `the set of strongly inaccessible cardinals is order-isomorphic to $\alpha$'.}

	We certainly will not use the full strength of \textsc{au}.
	At the cost of some awkward circumlocutions, one could even get away with \textsc{zfc} alone.
\end{nul}

\begin{dfn}
	Let $\delta$ be a strongly inaccessible cardinal.
	A set, group, simplicial set, category, ring, \textit{\& c.}, will be said to be \defn{$\delta$-small}\footnote{The adverb `essentially' is often deployed in this situation.} if it is equivalent (in whatever appropriate sense) to one that lies in $\VV_{\delta}$.
	We write
	\[
		\left.
		\begin{aligned}
    		\text{\defn{tiny}} & \\
    		\text{\defn{small}} &
  		\end{aligned}
  		\right\}
		\text{\ as shorthand for\ }
		\left\{
		\begin{aligned}
    		& \text{\defn{$\delta_0$-small}} \\
    		& \text{\defn{$\delta_1$-small.}}
  		\end{aligned}
  		\right.
	\]

	A category $ C $ is said to be \defn{locally $\delta$-small} if and only if, for any objects $x,y\in C$, the mapping space $\Map_C(x,y)$ is $\delta$-small.
	We write
	\[
		\left.
		\begin{aligned}
    		\text{\defn{locally tiny}} & \\
    		\text{\defn{locally small}} &
  		\end{aligned}
  		\right\}
		\text{\ as shorthand for\ }
		\left\{
		\begin{aligned}
    		& \text{\defn{locally $\delta_0$-small}} \\
    		& \text{\defn{locally $\delta_1$-small.}}
  		\end{aligned}
  		\right.
	\]
\end{dfn}

\begin{nul}
	For a strongly inaccessible cardinal $\delta$, we shall write $\Space_{\delta}$ for the category of $\delta$-small spaces and $\Cat_{\delta}$ for the category of $\delta$-small categories.
	The categories $\Space_{\delta_{\alpha}}$ and $\Cat_{\delta_{\alpha}}$ for the are $\delta_{\alpha+1}$-small and locally $\delta_{\alpha}$-small.
	We write
	\[
		\left.
		\begin{aligned}
    		\Space & \\
    		\Cat &
  		\end{aligned}
  		\right\}
		\text{\ as shorthand for\ }
		\left\{
		\begin{aligned}
    		& \Space_{\delta_1} \\
    		& \Cat_{\delta_1} \period
  		\end{aligned}
  		\right.
	\]
\end{nul}

\begin{nul}
	In the same vein, if $\delta$ is a strongly inaccessible cardinal, \emph{$\delta$-accessibility} of categories and functors and \emph{$\delta$-presentability} of categories will refer to accessibility and presentability with respect to some $\delta$-small cardinal.
	Please observe that a $\delta_{\alpha}$-accessible category is always $\delta_{\alpha+1}$-small and locally $\delta_{\alpha}$-small.
	We shall write
	\[
		\Pr_{\delta_{\alpha}}^L \subset \Cat_{\delta_{\alpha+1}} \respectively{\Pr_{\delta_{\alpha}}^R \subset \Cat_{\delta_{\alpha+1}}}
	\]
	for the subcategory whose objects are presentable categories and whose functors are left (resp., right) adjoints.
	We write
	\[
		\left.
		\begin{aligned}
    		\text{\defn{accessible}} & \\
    		\text{\defn{presentable}} & \\
    		\Pr^L & \\
    		\Pr^R &
  		\end{aligned}
  		\right\}
		\text{\ as shorthand for\ }
		\left\{
		\begin{aligned}
    		& \text{\defn{$\delta_1$-accessible}} \\
    		& \text{\defn{$\delta_1$-presentable}} \\
    		& \Pr_{\delta_1}^L \\
    		& \Pr_{\delta_1}^R
  		\end{aligned}
  		\right.
	\]

	Accordingly, a \defn{$\delta$-topos} is a left exact accessible localisation of a functor category $\Fun(C, \Space_{\delta})$ for some $\delta$-small category $ C $.
	We write \defn{topos} as a shorthand for \defn{$\delta_1$-topos}.
\end{nul}






\subsection{Sites and sheaves}

\begin{dfn}\label{dfn:site}
	A \defn{site} $(C,\tau)$ consists of a category $ C $ equipped with a Grothendieck topology $\tau$.
\end{dfn}

\begin{ntn}
	Let $\delta$ be a strongly inaccessible cardinal.
	We write
	\[
		\Sh_{\tau}(C)_{\delta} \subseteq \Fun(C^{\op},\Space_{\delta})
	\]
	for the full subcategory spanned by the sheaves on $ C $ with respect to the topology $\tau$.
	We write \smash{$ \Sh_{\tau}^{\hyp}(C)_{\delta} \subset \Sh_{\tau}(C)_{\delta} $} for the full subcategory spanned by the \textit{hypercomplete} sheaves.\footnote{For background on hypercompletness, see \cite[\HTTsec{6.5}]{HTT}.}
	In particular, we write $ \Sh_{\tau}(C) $ and \smash{$ \Sh_{\tau}^{\hyp}(C) $} as a shorthand for $ \Sh_{\tau}(C)_{\delta_1}$ and \smash{$ \Sh_{\tau}^{\hyp}(C)_{\delta_1} $}, respectively.
\end{ntn}

\begin{wrn}
	Let $(C,\tau)$ be a site.
	Assume that for some object $X\in C$, there does \textit{not} exist a tiny set of covering sieves of $X$ that is cofinal among all covering sieves.\footnote{So, in particular, $ C $ itself is not tiny.}
	Then the sheafification of a tiny presheaf on $ C $ (i.e., a presheaf $C^{\op} \to \Space_{\delta_0}$) might no longer be tiny.
	The point is that sheafification will involve a colimit over all covering sieves.
	As a consequence, the category $\Sh_{\tau}(C)_{\delta_0}$ of tiny sheaves on $ C $ is not $\delta_0$-topos.
	This is a perennial bugbear, for example, with the fpqc topology on the category of affine schemes.
	The sites $(C,\tau)$ with which we will be working suffer from this as well.

	Some authors simply elect never to sheafify a presheaf with respect to such topologies.
	However, in this article, we will be unable to avoid sheafification, and we do not wish to pass artificially to a subcategory of $ C $, so we will permit ourselves the luxury of `universe hopping':
	in our cases of interest, $ C $ will be small (but not tiny!), and so $\Sh_{\tau}(C)$ is a left exact localisation of $\Fun(C^{\op},\Space_{\delta_1})$ and thus a $\delta_1$-topos.

	In \cref{app:A}, we outline a proof that when the site is suitably accessible, then the sheafification of the small sheaves that arise in practise are again small.
	This is an adaptation of the strategy developed by Waterhouse \cite{MR0396578}.
	This gives a slightly more conservative way to deal with this issue.
\end{wrn}

\begin{dfn}\label{dfn:finitaryinftysite}
	A site $(C,\tau)$ is said to be \defn{finitary} if and only if $ C $ admits all finite limits, and, for every object $X\in C$ and every covering sieve $R \subseteq C_{/X}$, there is a finite subset $\{Y_i\}_{i\in I} \subseteq R$ that generates a covering sieve.
\end{dfn}

\begin{dfn}\label{dfn:presite}
	A \defn{presite} is a pair $(C,E)$ consisting of a category $ C $ along with a subcategory $E\subseteq C$ satisfying the following conditions.
	\begin{itemize}
		\item The subcategory $E$ contains all equivalences of $ C $.

		\item The category $ C $ admits finite limits, and $E$ is stable under base change.

		\item The category $ C $ admits finite coproducts, which are universal, and $E$ is closed under finite coproducts.
	\end{itemize} 
\end{dfn}

\begin{cnstr}\label{cnstr:presitefinitary}
	If $(C,E)$ is a presite, then there exists a topology $\tau_E$ on $ C $ in which the $\tau_E$-covering sieves are generated by finite families $\{V_i \to U\}_{i \in I}$ such that $\coprod_{i\in I}V_i \to U$ lies in $E$ \SAG{Proposition}{A.3.2.1}.
	The site $(C,\tau_E)$ is finitary.
	We simplify notation and write $\Sh_E(C) \subseteq \Fun(C^{\op},\Space_{\delta_1})$ for the full subcategory spanned by the small $\tau_E$-sheaves.
	Note that $\Sh_E(C)$ is a topos if $ C $ is small.

	If in addition the coproducts in $ C $ are disjoint, then a sheaf for $\tau_E$ valued in a category $D$ with all limits is a functor $F\colon C^{\op} \to D$ that carries finite coproducts in $ C $ to finite products in $D$, and for any morphism $ V \to U$ of $E$, the Čech nerve $ \Cech_{\ast}(V/U) \colon \mbfDelta_+^{\op} \to C $ induces an equivalence
	\[
		X(U) \equivalence \lim_{n \in \mbfDelta} X(\Cech_n(V/U)) 
	\]
	\SAG{Proposition}{A.3.3.1}.
	In this case, the topology $ \tau_E $ is subcanonical.
\end{cnstr}

\subsection{Accessible sheaves}\label{app:A}

Let $(C,\tau)$ be a site.
Assume that for some object $X\in C$, there does \textit{not} exist a tiny set of covering sieves of $X$ that is cofinal among all covering sieves.
Then the sheafification of a tiny presheaf on $ C $ (i.e., a presheaf $C^{\op} \to \Space_{\delta_0}$) might no longer be tiny.
The point is that sheafification will involve a colimit over all covering sieves.
As a consequence, the category $\Sh_{\tau}(C)_{\delta_0} \subset \Fun(C^{\op},\Space_{\delta_0})$ of tiny sheaves on $ C $ is not topos.
This becomes a concern, for example, for the fpqc site.
Here, we explain how one may identify conditions on a site that will allow us to sheafify \textit{accessible} presheaves without being forced to pass to a larger universe.
These conditions are satisfied by the fpqc site.
For the fpqc topology on discrete rings, this was observed by Waterhouse \cite{MR0396578}; our formulation only needs a small amount of extra care.

\begin{dfn}
	Let $\beta$ be a tiny regular cardinal. A presite $(C,E)$ is said to be \defn{$\beta$-accessible} if and only if the following conditions hold.
	\begin{itemize}
		\item
			Coproducts in $ C $ are disjoint.
		\item
			The opposite $C^{\op}$ is $\beta$-accessible.
			We write $C_{\beta} \subseteq C$ for the tiny category of $\beta$-cocompact objects (i.e., objects that are $\beta$-compact as objects of $C^{\op}$).
		\item
			Every morphism $X' \to X$ of $E$ can be exhibited as a limit of a diagram $\Lambda^{\op} \to \Fun(\Delta^1,E \cap C_{\beta})$ in which $\Lambda$ is $\beta$-filtered.
	\end{itemize}

	We say that a small presite $(C,E)$ is \defn{accessible} if and only if $ (C,E) $ is $\beta$-accessible for some tiny regular cardinal $\beta$.
\end{dfn} 

\begin{nul}
	Let $\beta$ be a tiny regular cardinal, and let $(C,E)$ be a $\beta$-accessible presite.
	Write $E_{\beta} \coloneq E \cap C_{\beta}$; then $(C_{\beta},E_{\beta})$ is a tiny presite (in which coproducts are still disjoint).
	Consequently, $\Sh_{E_{\beta}}(C_{\beta})_{\delta_0}$ is a $\delta_0$-topos.
\end{nul}

\begin{prp}
	Let $\beta$ be a tiny regular uncountable cardinal, and let $(C,E)$ be a $\beta$-accessible presite.
	Let $f \colon C_{\beta}^{\op} \to \Space_{\delta_0}$ be a functor, and let $F\colon C^{\op}\to\Space_{\delta_0}$ be the left Kan extension of $ f $.
	Then $f$ is a $\tau_{E_{\beta}}$-sheaf if and only if $F$ is a $\tau_E$-sheaf.
\end{prp}

\begin{proof}
	Since every object of $C^{\op}$ is a $\beta$-filtered colimit of objects of $C_{\beta}^{\op}$, it follows that $f$ preserves finite products if and only if $F$ does.

	If $F$ is a sheaf, then the description above ensures that $f$ is a sheaf as well.

	Let $ e \colon V \to U$ be a morphism of $ E $, and let $ \Cech_{\ast}(e) \colon \mbfDelta_+^{\op} \to C $ denote the Čech nerve of $ e $.
	Exhibit $ e $ as a limit
	\[
		\lim_{\alpha\in\Lambda^{\op}} V_{\alpha} \to \lim_{\alpha\in\Lambda^{\op}} U_{\alpha} \comma
	\]
	where $ \Lambda $ is $\alpha$-filtered, and each \smash{$ e_{\alpha} \colon \fromto{V_\alpha}{U_{\alpha}} $} lies in $ E_{\beta} $; in particular each object $ \Cech_n(e_{\alpha}) $ is $\beta$-cocompact.
	Then \smash{$ \Cech_n(e) \simeq \lim_{\alpha\in\Lambda^{\op}} \Cech_{n}(e_{\alpha}) $}, and the map \smash{$X(U)\to\lim_{n\in\mbfDelta} X (\Cech_n(e))$} can be exhibited as the colimit
	\[
		\colim_{\alpha\in\Lambda} X(U_{\alpha}) \to \lim_{n\in\mbfDelta}\colim_{\alpha\in\Lambda} X(\Cech_n(e_{\alpha})) \period
	\]
	Since $\Lambda$ is $\beta$-filtered and $ \beta $ is uncountable, the colimit commutes with the limit, and so the map $X(U)\to\lim_{n\in\mbfDelta}X(V_n)$ is the colimit of a diagram of equivalences
	\[
		X(U_{\alpha}) \equivalence \lim_{n\in\mbfDelta} X(\Cech_n(e_{\alpha})) \comma
	\]
	hence an equivalence.
\end{proof}

\begin{nul}
	Let $(C,E)$ be an $\omega$-accessible presite.
	If $N$ is a natural number, then a functor $f\colon C_{\beta}^{\op}\to\tau_{\leq N}\Space_{\delta_0}$ is a $\tau_{E_{\beta}}$-sheaf if and only if its left Kan extension
	\begin{equation*}
		F\colon C^{\op}\to\tau_{\leq N}\Space_{\delta_0}
	\end{equation*}
	is a $\tau_E$-sheaf.
	The truncatedness assumption ensures that the limit over $\mbfDelta$ can be replaced with a limit over the full subcategory \smash{$ \mbfDelta_{\leq N+1} $} of totally ordered finite sets of cardinality at most $ N+2 $, which is finite. 
	This permits us to commute the filtered colimit past the totalisation.
\end{nul}

\begin{cor}
	Let $\beta$ be a tiny, regular, uncountable cardinal, and let $(C,E)$ be a $\beta$-accessible presite.
	The left Kan extension defines an equivalence of categories between the topos $\Sh_{\tau_{E_{\beta}}}(C_{\beta})_{\delta_0}$ and the full subcategory of $\Sh_{\tau_E}(C)_{\delta_0}$ spanned by the $\beta$-accessible sheaves.
\end{cor}

\begin{ntn}
	If $(C,E)$ is an accessible presite, then we write
	\begin{equation*}
		\Sh_E^{\acc}(C)_{\delta_0} \subseteq \Fun(C^{\op},\Space_{\delta_0})
	\end{equation*}
	for the full subcategory spanned by the accessible sheaves.
	More generally, if $D$ is any $\delta_0$-presentable category, then $\Sh_E^{\acc}(C; D) \subseteq \Fun(C^{\op},D)$ is the full subcategory spanned by the accessible sheaves.
\end{ntn}

Now we may see that sheafification of accessible functors does not increase the size of the universe.

\begin{cor}
	Let $(C,E)$ be an accessible presite, and let $D$ be a $\delta_0$-presentable category.
	Then $\Sh_E^{\acc}(C; D)$ is a left exact localisation of the category $\Fun^{\acc}(C^{\op},D)$ of accessible functors $C^{\op} \to D$.
\end{cor}

\begin{exm}
	If $ C $ is a tiny regular disjunctive category, then $\Pro^{\delta_0}(C)$ is an accessible presite with its effective epimorphism topology.
\end{exm}

\begin{wrn}
	If $(C,E)$ is an accessible presite that is not tiny, please observe that $\Sh_E^{\acc}(C)_{\delta_0}$ cannot be expected to be a $\delta_0$-topos, or even $\kappa$-accessible with respect to a tiny cardinal $\kappa$.
	It is however locally tiny, and it does have many of the good features enjoyed by $\delta_0$-topoi.
	For convenience, we formalise the situation.
\end{wrn}

\begin{dfn}
	Let $D$ be an accessible category, and let
	\[
		L \colon \Fun^{\acc}(D, \Space_{\delta_0}) \to \XX \subseteq \Fun^{\acc}(D, \Space_{\delta_0})
	\]
	be a localisation.
	For any small regular cardinal $\alpha$, if $D$ is $\alpha$-accessible, then let us write $\XX_{\alpha}$ for the essential image of $L$ restricted to $\Fun(D_{\alpha}, \Space_{\delta_0}) \simeq \Fun^{\alpha\text{-}\acc}(D, \Space_{\delta_0})$.
	Equivalently, $\XX_{\alpha}$ is the intersection
	\[
		\XX_{\alpha} \simeq \XX \cap \Fun(D_{\alpha}, \Space_{\delta_0}) \period
	\]
	inside $\Fun^{\acc}(D, \Space_{\delta_0})$.
	We shall say that the localisation functor $L$ is \defn{macroaccessible} if for any small cardinal $\beta$, there exists a small regular cardinal $\alpha>\beta$ such that $D$ is $\alpha$-accessible, and $L$ restricts to an accessible functor
	\[
		L_{\alpha} \colon \Fun(D_{\alpha}, \Space_{\delta_0}) \to \XX_{\alpha} \subseteq \Fun^{\acc}(D_{\alpha}, \Space_{\delta_0}) \period
	\]

	A \defn{macropresentable} category is a category $\XX$ such that there exists an accessible category $D$ and a macroaccessible localisation
	\[
		L \colon \Fun^{\acc}(D, \Space_{\delta_0}) \to \Fun^{\acc}(D, \Space_{\delta_0})
	\]
	whose essential image is equivalent to $\XX$.

	A \defn{macrotopos} is a category $\XX$ such that there exists an accessible category $D$ and a left exact, macroaccessible localisation
	\[
		L \colon \Fun^{\acc}(D, \Space_{\delta_0}) \to \XX \subseteq \Fun^{\acc}(D, \Space_{\delta_0}) \period
	\]
\end{dfn}

\begin{nul}
	If $\XX$ is a macropresentable category, then $ \XX $ is the macroaccessible localisation of $\Fun^{\acc}(D, \Space_{\delta_0})$ for an accessible category $D$; let us write
	\begin{equation*}
		L \colon \Fun^{\acc}(D, \Space_{\delta_0}) \to \XX \subseteq \Fun^{\acc}(D, \Space_{\delta_0})
	\end{equation*}
	for the localisation functor.
	If $\alpha<\beta$ are regular cardinals with the properties that $D$ is both $\alpha$- and $\beta$-accessible and that $L$ restricts to accessible functors
	\[
		L_{\alpha} \colon \Fun(D_{\alpha}, \Space_{\delta_0}) \to \Fun^{\acc}(D_{\alpha}, \Space_{\delta_0}) \andeq L_{\beta} \colon \Fun(D_{\beta}, \Space_{\delta_0}) \to \Fun(D_{\beta}, \Space_{\delta_0}) \comma
	\]
	then we have an inclusion $\XX_{\alpha} \subseteq \XX_{\beta}$.
	The macropresentable category $ \XX $ is the $\delta_1$-small filtered colimit of the presentable categories $\XX_{\alpha}$ under fully faithful left adjoints.
	Similarly, if $\XX$ is a macrotopos, then the $\XX_{\alpha}$ are topoi, and so $\XX$ is a $\delta_1$-small filtered union of topoi under fully faithful left exact left adjoints.
\end{nul}

\begin{exm}
	If $(C,E)$ is an accessible presite, then $\Sh^{\acc}_E(C)_{\delta_0}$ is a macrotopos.
\end{exm}


\section{Pyknotic objects}


\subsection{Pyknotic sets}

\begin{nul}\label{nul:topologicalntn}
	Let $\TSpc$ denote the category of tiny topological spaces.
	Write
	\[
		\Comp \subset \TSpc
	\]
	for the full subcategory spanned by the \emph{compacta} -- i.e., tiny compact hausdorff topological spaces.
	
	We write $ \beta \colon \fromto{\TSpc}{\Comp} $ for the left adjoint to the inclusion, given by Stone--Čech compactification.

	The category $\Comp$ can be identified with the category of $\beta$-algebras on $\Set_{\delta_0}$, where $\beta \colon \Set_{\delta_0} \to \Set_{\delta_0}$ is the ultrafilter monad \cite[Chapter III, \S2.4]{MR861951}.
\end{nul}

\begin{nul}
	Since the category $ \Comp $ of compacta is a $ 1 $-pretopos, $ \Comp $ comes equipped with the \defn{effective epimorphism topology}; a collection of morphisms $ \{U_i \to U\}_{i \in I} $ is a cover if and only if there exists a finite subset $ I_0 \subset I $ such that the map
	\begin{equation*}
		\surjto{\coprod_{i \in I_0} U_i}{U}
	\end{equation*}
	is a surjection (=effective epimorphism in $ \Comp $).

	Note that \Cref{cnstr:presitefinitary} gives a complete characterisation of sheaves on $ \Comp $; see also \cite[Proposition B.5.5]{Ultracategories}.
\end{nul}



\begin{dfn}\label{def:PykS}
	The category of \defn{pyknotic sets} is the category
	\begin{equation*}
		\Pyk(\Set) \coloneq \Sheff(\Comp; \Set)
	\end{equation*}
	of small sheaves of sets on $ \Comp $ with respect to the effective epimorphism topology.
\end{dfn}

\begin{nul}\label{nul:coherenceofPykS}
	The category $\Pyk(\Set)$ is a coherent $ 1 $-topos.
	By the classification theorem for coherent $ 1 $-topoi \cite[Theorem C.6.5]{Ultracategories}, the coherent objects of $\Pyk(\Set)$ are exactly the compacta, regarded as representables.
\end{nul}

\begin{ntn}
	The $ 1 $-category $ \CG $ of \defn{compactly generated topological spaces} is the smallest full subcategory of the category \smash{$ \TSpc_{\delta_1} $} of \emph{small} topological spaces containing $ \Comp $ and closed under small colimits.
	In particular, $\CG$ is a colocalisation of \smash{$\TSpc_{\delta_1}$}.
\end{ntn}

\begin{exm}\label{exm:CGembeds}
	Let $X$ be a small topological space.
	Then the functor
	\[
		\Mor_{\TSpc_{\delta_1}}(-, X) \colon \Comp^{\op} \to \Set
	\]
	is pyknotic set.
	We can endow the underlying set of $X$ with the induced topology with respect to the class of continuous morphisms from compacta.
	This is as coarse as the topology on $X$, and coincides with the topology on $X$ if and only if $X$ is compactly generated.

	In other words, the Yoneda embedding extends to a functor $ \yo \colon \TSpc \to \Pyk(\Set)$ with a left adjoint defined by left Kan extension of the inclusion $ \incto{\Comp}{\TSpc} $ along $ \incto{\Comp}{\Pyk(\Set)} $.
	The counit of this adjunction is a homeomorphism on compactly generated topological spaces, and so the Yoneda embedding defines a fully faithful functor from compactly generated topological spaces into pyknotic sets;
	this expresses the category $ \CG $ as a \textit{localisation} of $ \Pyk(\Set) $.

	This is one important way in which topological spaces are different from pyknotic sets: compactly generated topological spaces are not stable under colimits in $\Pyk(\Set)$.

	Notationally, we'll often ignore the distinction between a compactly generated topological spaces and its corresponding pyknotic set. 
\end{exm}

Even though compactly generated topological spaces aren't closed under colimits in $ \Pyk(\Set) $, they are closed under a certain class of colimits:

\begin{lem}\label{lem:colimitinPykSofnicetopspaces}
	Let $ X_0 \to X_1 \to \cdots $ be a sequence of compactly generated topological spaces.
	Assume that the colimit $ \colim_n X_n $ in \smash{$ \TSpc_{\delta_1} $} is a $ T_1 $ topological space.
	Then the natural morphism
	\begin{equation*}
		\fromto{\textstyle\colim_{n \geq 0} \yo(X_n)}{\yo(\textstyle\colim_{n \geq 0} X_n)}
	\end{equation*}
	is an equivalence in $ \Pyk(\Set) $.
\end{lem}

\begin{proof}
	For each compactum $ K $, the object $ \yo(K) \in \Pyk(\Set) $ is compact \cite[Lemma 5.8.2]{exodromy}, so we have isomorphisms
	\begin{align*}
		\Map_{\Pyk(\Set)}(\yo(K),\textstyle\colim_{n \geq 0} \yo(X_n)) &\isomorphic \textstyle\colim_{n \geq 0} \Map_{\Pyk(\Set)}(\yo(K),\yo(X_n)) \\
		&\isomorphic \textstyle\colim_{n \geq 0} \Map_{\CG}(K,X_n) \\
		&\isomorphic \Map_{\CG}(K,\textstyle\colim_{n \geq 0} X_n) \\
		&\isomorphic \Map_{\Pyk(\Set)}(\yo(K),\yo(\textstyle\colim_{n \geq 0} X_n)) \period
	\end{align*}
	The second isomorphism is by the full faithfulness of $ \yo \colon \incto{\CG}{\Pyk(\Set)} $.
	The third isomorphism is by \cite[Appendix A, Lemma 9.4]{LewisThesis}, which states that for any map from a compactum $ f \colon \fromto{K}{\colim_{n \geq 0} X_n} $, the image of $ f $ factors through some $ X_n $.
\end{proof}

\begin{wrn}
	Note that \cite[Appendix A, Lemma 9.4]{LewisThesis} used in the proof of \Cref{lem:colimitinPykSofnicetopspaces} does not hold for more general filtered colimits: the unit interval is the filtered colimit of all of its countable subspaces, but the identity map does not factor through a countable subspace.
\end{wrn}

\begin{exm}
	The category $ \Pyk(\Set) $ is compactly generated and the Yoneda embedding $ \incto{\Comp}{\Pyk(\Set)} $ carries compacta to compact objects of the category $ \Pyk(\Set) $ \SAG{Corollary}{A.2.3.2}.
	Thus the Yoneda embedding extends to a fully faithful embedding
	\begin{equation*}
		\incto{\Ind(\Comp)}{\Pyk(\Set)} 
	\end{equation*}
	\HTT{Proposition}{5.3.5.11}.
	Regarding profinite sets as Stone topological spaces under Stone duality, we thus obtain an embedding
	\begin{equation*}
		\incto{\Ind(\Pro(\Setfin))}{\Pyk(\Set)} \period
	\end{equation*}
	
	Indprofinite sets and extensions to indpro$\cdots$indprofinite sets have been exploited by Kato in studying higher local fields \cite{MR1804933}, as well as Mazel-Gee--Peterson--Stapleton in homotopy theory \cite[\S2]{MR3402340}.
	In particular, local fields of dimension at most $ 1 $ may be understood in terms of indprofinite sets.
\end{exm}

\begin{exm}
	Since a compactum has a unique uniformity compactible with its topology \cite[Chapter II, \S4, \P1, Theorem 1]{MR1726779}, any uniform space $ U $ defines a pyknotic set by the assignment $ \goesto{K}{\Mor_{\Unif}(K,U)} $.
	This restricts to a fully faithful embedding from the full subcategory of \defn{compactly generated} uniform spaces -- those uniform spaces $ U $ for which a set-map $ \fromto{U}{U'} $ to another uniform space $ U' $ is uniformly continuous if and only if for every uniformly continuous map $ \fromto{K}{U} $ from a compactum, the composite $ \fromto{K}{U'} $ is continuous. 
\end{exm}


\subsection{Pyknotic spaces}

\begin{nul}\label{nul:descriptionsofPykS}
	Define two full subcategories
	\[
		\EStn \subset \Stn \subset \Comp
	\]
	as follows:
	\begin{itemize}
		\item
			$\Stn$ is spanned by the \emph{Stone} topological spaces -- i.e., tiny compact hausdorff spaces that are totally disconnected;
		\item
			$\EStn$ is spanned by the \emph{Stonean} topological spaces -- i.e., tiny compact hausdorff spaces that are extremally disconnected.
	\end{itemize}
	All of these categories are small but not tiny.

	Under Stone duality, the category $\Stn$ can be identified with the category \smash{$\Pro(\Set^{\fin})$} of profinite sets.
	By Gleason's theorem, the category $\EStn$ can be identified with the category of projective objects of $\Comp$ \cites{MR0121775}[Chapter III, \S3.7]{MR861951}; equivalently, a topological space is Stonean if and only if it can be exhibited as the retract of $\beta(S)$ for some (tiny) set $S$.

	Restriction of presheaves defines equivalences of $ 1 $-categories
	\begin{equation*}
		\Pyk(\Set) \coloneq \Sheff(\Comp;\Set) \equivalence \Sheff(\Stn;\Set) \equivalence \Sheff(\EStn;\Set) \period
	\end{equation*}
	These equivalences follow from the from the following three facts:
	\begin{itemize}
		\item If $ (C,\tau) $ is a $ 1 $-site and $ C' \subset C $ is a \textit{basis} for the topology $ \tau $ \cite[Definition B.6.1]{Ultracategories}, then $ \tau $ restricts to a topology $ \tau' $ on $ C' $ and restriction defines an equivalence of $ 1 $-categories
		\begin{equation*}
			\equivto{\Sh_{\tau}(C;\Set)}{\Sh_{\tau}(C';\Set)} \rlap{\ ;}
		\end{equation*}
		see \cite[Propositions B.6.3 \& B.6.4]{Ultracategories}.

		\item A Stone space $ S $ is extremally disconnected if and only if $ S $ is a retract of the Stone--Čech compactification of a discrete space.

		\item For every compactum $ X $, there is a natural surjection $ \surjto{\beta(X^{\delta})}{X} $ from the Stone--Čech compactification of the discrete space $ X^{\delta} $ with underlying set $ X $ to $ X $ (cf. \cite[Remark  2.8]{Scholze:diamonds}).
		Hence the subcategories $ \Stn \subset \Comp $ and $ \EStn \subset \Comp $ are bases for the effective epimorphism topology on $ \Comp $. 
	\end{itemize}
\end{nul}

\begin{wrn}\label{wrn:PykS}
	Since the $ 1 $-sites $ \Comp $ and $ \Stn $ have finite limits and the inclusion $ \incto{\Stn}{\Comp} $ preserves finite limits, from \Cref{nul:descriptionsofPykS} we deduce that restriction defines an equivalence of $ 1 $-localic topoi
	\begin{equation*}
		\equivto{\Sheff(\Comp)}{\Sheff(\Stn)} \period
	\end{equation*}

	However, as pointed out to us by Dustin Clausen and Peter Scholze, since the $ 1 $-site $ \EStn $ of Stonean spaces does not have finite limits, restriction only defines an equivalence
	\begin{equation*}
		\equivto{\Sheff^{\hyp}(\Comp)}{\Sheff^{\hyp}(\EStn)}
	\end{equation*}
	on topoi of \textit{hypersheaves}.

	The topos $ \Sheff(\EStn) $ is in fact already hypercomplete (\Cref{cor:PykSashypercomp}), whence we obtain an equivalence
	\[
		\Sheff^{\hyp}(\Comp) \simeq \Sheff(\EStn)
	\]
	but $\Sheff(\Comp)$ is not hypercomplete, so it remains different.
\end{wrn}

\begin{dfn}
	A \defn{pseudopyknotic space} is a sheaf on $\Comp$ for the effective epimorphism topology.
	We write
	\begin{equation*}
		\PsiPyk(\Space) \coloneq \Sheff(\Comp)
	\end{equation*}
	for the category of pseudopyknotic spaces.
	A \defn{pyknotic space} is a hypersheaf on $\Comp$.
	We write
	\begin{equation*}
		\Pyk(\Space) \coloneq \Sheff^{\hyp}(\Comp)
	\end{equation*}
	for the category of pyknotic spaces.
\end{dfn}

\begin{nul}
	Equivalently, as explained above, pyknotic spaces are sheaves on the site of Stonean topological spaces.
\end{nul}

\begin{cnstr}
	For any compactum $K$, there is a \defn{standard free resolution}\footnote{Elsewhere called the \textit{bar construction}.} of $K$, regarded as an algebra for the ultrafilter monad $\beta$, \textit{viz}.,
	\[
		C^{\beta}_{\ast}(K)\coloneq\left[\augsimplicial{\beta^3(K)}{\beta^2(K)}{\beta(K)}{K}\right]\comma
	\]
	so that $ \beta^{n+1}(K) $ is the Stone--Čech compactification of the discrete space with underlying set $ \beta^n(K) $.
	The standard free resolution is a hypercovering of $K$ in $\Comp$ by Stonean topological spaces.
\end{cnstr}

\begin{prp}
	The following are equivalent for a pseudopyknotic space $X$.
	\begin{itemize}
		\item $X$ is pyknotic.
		\item $X$ is right Kan extended from the subcategory $\EStn \subset \Comp$.
		\item For any compactum $K$, the augmented cosimplicial space
		\[
			\augcosimplicial{X(K)}{X(\beta(K))}{X(\beta^2(K))}{X(\beta^3(K))}
		\]
		exhibits $X(K)$ as the limit $\lim_{\mbfDelta}X(C^{\beta}_{\ast}(K))$.
	\end{itemize}
\end{prp}

\begin{wrn}
	Not every pseudopyknotic space is pyknotic.
\end{wrn}

\begin{exm}
	We have already seen that compactly generated topological spaces and compactly generated uniform spaces embed fully faithfully into pyknotic sets;
	consequently, they embed into pyknotic spaces as well.

	Furthermore, since the inclusion of $ 0 $-truncated objects in a topos preserves filtered colimits, \Cref{lem:colimitinPykSofnicetopspaces} shows that the embedding $ \incto{\CG}{\Pyk(\Space)} $ commutes with colimits of sequences whose colimit is a $ T_1 $ topological space.
\end{exm}

\begin{nul}\label{nul:PykSasnonablianderived}
	Since Stonean spaces are projective objects of $ \Comp $, the Čech nerve of any surjection in $ \EStn $ is a split simplicial object, so a functor $ F \colon \fromto{\EStn^{\op}}{\Space} $ is a sheaf with respect to the effective epimorphism topology if and only if $ F $ carries coproducts in $ \EStn $ to products in $ \Space $.
	That is to say, the category $ \Sheff(\EStn) $ is the nonabelian derived category\footnote{See \cite[\HTTsec{5.5.8}]{HTT} for more on nonabelian derived categories.} $ \PP_{\Sigma}(\EStn)$ of the category $ \EStn $ Stonean topological spaces.

	From any Stone topological space one may extract the Boolean algebra of clopens;
	Stone duality is the assertion that this defines an equivalence between $\Stn$ and the opposite of the category $\Bool$ of Boolean algebras.
	This equivalence then restricts to an equivalence between $\EStn$ and the opposite of the category $\Boolcomp$ of complete Boolean algebras.
	Consequently, a pyknotic object of $D$ may be understood as a functor $\Boolcomp \to D$ that preserves finite products.
\end{nul}

\begin{nul}
	Since finite products commute with sifted colimits in $ \Space $, we see that
	\begin{equation*}
		\Pyk(\Space) \subset \Fun(\EStn^{\op},\Space) 
	\end{equation*}
	is closed under sifted colimits.
	In particular, geometric realisations of simplicial pyknotic spaces are computed in $ \Fun(\EStn^{\op},\Space) $. 
\end{nul}

\begin{exm}\label{exm:quotientsarequotients}
	As a consequence, we find that it is relatively easy to form \defn{quotient pyknotic structures}.
	For example, if $X$ is a pyknotic set and $R \subset X \times X$ is an equivalence relation thereupon, then the quotient $X/R$ can be computed objectwise on Stonean topological spaces:
	\[
		(X/R)(K) \simeq X(K) / R(K)\period
	\]

	In a similar vein, if $X_{\ast}$ is a simplicial pyknotic space, then its realisation can be computed objectwise on Stonean topological spaces:
	\[
		|X_{\ast}|(K) \simeq |X_{\ast}(K)| \period
	\]
\end{exm}

\begin{cnstr}\label{nul:PykSislocal}
	The global sections functor $ \Gammalowerstar \colon \fromto{\Pyk(\Space)}{\Space} $ is given by evaluation at the one-point compactum $ \pt $.
	For any pyknotic space $X$, we call $\Gammalowerstar(X)$ the \defn{underlying space} of $X$.
	When there's no possibility of confusion, we simply write $X$ for $\Gammalowerstar(X)$.

	Left adjont to this is the constant sheaf functor $ \Gammaupperstar \colon \fromto{\Space}{\Pyk(\Space)}$ that carries a space $Y$ to what we will call the \defn{discrete pyknotic space}
	\[
		Y^{\disc} \coloneq \Gammaupperstar(Y)
	\]
	attached to $Y$.
	
	The underlying space functor $ \Gammalowerstar $ also admits a \textit{right} adjoint $ \Gammauppershriek \colon \fromto{\Space}{\Pyk(\Space)} $: for $ X \in \Space $ the sheaf $ \Gammauppershriek(X) \colon \fromto{\Comp^{\op}}{\Space} $ is given by the assignment
	\begin{equation*}
		\goesto{K}{\prod_{k \in |K|} X} \comma
	\end{equation*}
	i.e., the product of copies of $ X $ indexed by the underlying \textit{set} of the compactum $ K $.
	For any space $X$, we call
	\[
		X^{\indisc} \coloneq \Gammauppershriek(X)
	\]
	the \defn{indiscrete pyknotic space} attached to $X$.

	The composite $ \Gammalowerstar \Gammauppershriek \colon \fromto{\Space}{\Space} $ is equivalent to the identity, so the indiscrete functor $ \Gammauppershriek $ is fully faithful, whence so is the discrete functor $ \Gammaupperstar \colon \fromto{\Space}{\Pyk(\Space)} $.
	In the language of \cite[Definition 7.2.2]{exodromy}, the topos $ \Pyk(\Space) $ is \textit{local} with centre $ \Gammauppershriek $.
	In particular, $ \Pyk(\Space) $ has homotopy dimension $ 0 $ \cite[Lemma 7.2.5]{exodromy}.

	Accordingly, a pyknotic space in the essential image of $\Gammaupperstar$ is said to be \defn{discrete}, a pyknotic space in the essential image of $\Gammauppershriek$ is said to be \defn{indiscrete}.
\end{cnstr}

\begin{nul}
	In particular, note that if $X$ is a presheaf $\Comp^{\op} \to \Space$, then its hypersheafification $X^{+}$ has the same underlying set.
	That is, $\Gammalowerstar(X) \to \Gammalowerstar(X')$ is an equivalence:
	indeed, for any space $Y$, the map
	\[
		\Map(X^{+}, \Gammauppershriek(Y)) \simeq \Map(\Gammalowerstar(X^{+}), Y) \to \Map(\Gammalowerstar(X), Y) \simeq \Map(X, \Gammauppershriek(Y))
	\]
	is an equivalence, since $\Gammauppershriek(Y)$ is a sheaf.
\end{nul}

\begin{exm}
	For any finite set $J$, the discrete pyknotic set $J^{\disc}$ is the sheaf $K \mapsto \Map(J, K)$ represented by $J$.
	If $\{J_{\alpha}\}_{\alpha\in\Lambda}$ is an inverse system of finite sets, then the limit
	\[
		\lim_{\alpha\in\Lambda} J_{\alpha}^{\disc}
	\]
	is the sheaf represented by the Stone topological space $\lim_{\alpha\in\Lambda} J_{\alpha}$;
	this is not discrete.

	In particular, the discrete functor $\Gammaupperstar$ does not preserve limits, and so the topos $\Pyk(\Space)$ is -- by design -- not \defn{cohesive} in the sense of Schreiber \cite[Definition 3.4.1]{Schreiber:cohesive}.
\end{exm}

\begin{nul}
	The point $ \Gammauppershriek $ of the topos $\Pyk(\Space) $ admits a description coming from logic: $ \Gammauppershriek $ is the point induced by the morphism of $ 1 $-pretopoi $ \fromto{\Comp}{\Set} $ given by the forgetful functor (see \SAG{Proposition}{A.6.4.4}).
\end{nul}

\begin{nul}
	The point $ \Gammauppershriek$ of the topos $\Pyk(\Space) $ also admits a geometric description.
	Let $ k $ be a separably closed field.
	Then the hypercompletion of the proétale topos $ \Spec(k)_{\proet} $ of $ \Spec(k) $ is equivalent to $ \Pyk(\Space) $ (\Cref{exm:profinitespacesembedintoPykS,exm:proetasPykalg}).
	Every geometric point of a scheme defines a point of its proétale topos \stacks{0991}, so the essentially unique geometric point $ x $ of $ \Spec(k) $ defines a point
	\begin{equation*}
		\xlowerstar \colon \fromto{\Space}{\Spec(k)_{\proet}} \period
	\end{equation*}
	Under the identification $ \Spec(k)_{\proet}^{\hyp} \equivalent \Pyk(\Space) $, the point $ \Gammauppershriek $ is equivalent to $ \xlowerstar $.
\end{nul}

\begin{wrn}
	However, the centre $ \Gammauppershriek \colon \incto{\Space}{\Pyk(\Space)} $ is \textit{not} the only point of the topos $ \Pyk(\Space) $.
	For any topological space $X$, we have a pyknotic set $P_X$ that carries $K$ to the set of continuous maps $K \to X$, where the locally constant maps have been identified to a point:
	\[
		P_S(K) \coloneq \Map^{\cts}(K, X)/\Map^{\lc}(K, X) \period
	\]
	If $X$ is nonempty, then the pyknotic set $P_X$ has underlying set $\pt$;
	thus if $X$ is neither empty nor $\pt$, then $P_X$ is a nontrivial pyknotic structure on the point.
	See \stacks{0991} and also \Cref{cor:PykSenoughpoints}.
\end{wrn}

\begin{nul}
	Let $X$ be a space (respectively, a set).
	The \defn{category of pyknotic structures on $X$} is the fibre of the functor $\Gammalowerstar \colon \Pyk(\Space) \to \Space$ (resp., $\Gammalowerstar \colon \Pyk(\Set) \to \Set$).

	This category admits an initial object $X^{\disc}$ and a terminal object $X^{\indisc}$.
	Furthermore, the category of pyknotic structures on $X$ has all tiny limits and colimits.

	However, unlike the category of topologies on a set, it is not a poset.
	For example, any permutation of a nonempty set $S$ induces a automorphism of $P_S$.
\end{nul}

\begin{cnstr}
	Let $X$ be a space, and let $Y$ be a pyknotic space.
	For any map $f \colon X \to \Gammalowerstar(Y)$, there is a terminal object in the category of pyknotic structures on $X$ over $Y$;
	explicitly, this is the pullback
	\[
		X^f \coloneq \Gammauppershriek(X) \times_{\Gammauppershriek\Gammalowerstar(Y)} Y \period
	\]
	We call this the \defn{pyknotic structure on $X$ induced by $f$}.

	Dually, for any map $g \colon \Gammalowerstar Y \to X$, there is an initial object in the category of pyknotic structures on $X$ under $Y$;
	explicitly, this is the pushout
	\[
		X_g \coloneq \Gammaupperstar(X) \cup^{\Gammaupperstar\Gammalowerstar(Y)} Y\period
	\]
	We call this the \defn{pyknotic structure on $X$ coinduced by $g$}.
\end{cnstr}

\begin{exm}
	Let $Y$ be a topological space, and let $X \to Y$ be a map of sets.
	View $Y$ as a pyknotic set.
	Then the induced pyknotic structure on $X$ coincides with the pyknotic structure attached to the induced topology on $X$.
\end{exm}



\subsection{Pyknotic objects}

In the previous subsection, we reformulated the definition of a pyknotic space in terms of finite-product-preserving presheaves on Stonean spaces.
We can thus define pyknotic objects in any category with finite products. 

\begin{dfn}
	Let $D$ be a category with all finite products.
	A \emph{pyknotic object} of $D$ is a functor $\EStn^{\op} \to D$ that carries finite coproducts of Stonean topological spaces to products in $D$.
	We write
	\begin{equation*}
		\Pyk(D) \subseteq \Fun(\EStn^{\op},D)
	\end{equation*}
	for the full subcategory spanned by the pyknotic objects.
\end{dfn}

\begin{wrn}
	Since $\EStn$ is small but not tiny, $\Pyk(D)$ is not generally locally tiny, even if $D$ is.
	However, if $D$ is locally small, then $\Pyk(D)$ is locally small.
\end{wrn}

To correct this issue without large cardinals, Clausen and Scholze opt for the following.

\begin{dfn}[Clausen--Scholze]
	If $D$ is a $\delta_0$-accessible category, then a pyknotic object of $D$ is \defn{condensed} (relative to the tiny universe) if and only if its right Kan extension to $ \Stn $ is a $\delta_0$-accessible sheaf.
\end{dfn}

\begin{wrn}
	The indiscrete pyknotic set $Y^{\indisc}$ attached to a set $Y$ is condensed if and only if $Y$ has cardinality $\leq 1$.
\end{wrn}

\begin{nul}\label{nul:pykastensor}
	If $D$ is a category with all small limits, then $\Pyk(D)$ can be identified with the category of functors $\Pyk(\Space)^{\op} \to D$ that carry small colimits of $\Pyk(\Space)$ to limits in $D$.
	In particular, if $ D $ is a presentable category, then $ \Pyk(D)$ is the tensor product of presentable categories $\Pyk(\Space) \tensor D $. 
	In particular, if $ \XX $ is a topos, then $ \Pyk(\XX) $ is a topos.
\end{nul}

\begin{cnstr}\label{cnstr:discindisc}
	If $D$ is a presentable category, then we may tensor the left adjoints in \Cref{nul:PykSislocal} with $D$ to construct a chain of adjoints
	\begin{equation*}
		\begin{tikzcd}[sep=5.5em]
			\Pyk(D) \arrow[r,"\Gammalowerstar"{xshift=0ex,description}] & D \period \arrow[l, hooked', shift left=1.5ex, "\Gammauppershriek"{xshift=2.5ex,description}] \arrow[l, hooked', shift right=1.5ex, "\Gammaupperstar"{xshift=-2.5ex,description}]
		\end{tikzcd}
	\end{equation*}

	For any object $X$ of $D$, then when there's no possibility of confusion, we write simply $X$ for $\Gammalowerstar(X)$.
	For any pyknotic object $Y$ of $D$, we write
	\[
		Y^{\disc} \coloneq \Gammaupperstar(Y)
	\]
	for the \defn{discrete pyknotic object} attached to $Y$, and we write
	\[
		Y^{\indisc} \coloneq \Gammauppershriek(Y)
	\]
	for the \defn{indiscrete pyknotic object} attached to $Y$.
\end{cnstr}

\begin{exm}
	If $G$ is a topological group, then we may regard $G$ as a pyknotic group that carries a compactum $K$ to $\Map^{\cts}(K, G)$.
	This defines a functor from topological groups to pyknotic groups, which preserves limits and is fully faithful on compactly generated topological groups.

	In particular, if $\{G_{\alpha}\}_{\alpha\in \Lambda^{\op}}$ is an inverse system of groups, the inverse limit
	\[
		\lim_{\alpha\in \Lambda^{\op}} G_{\alpha}^{\disc}
	\]
	will generally not be discrete.
	For instance, the discrete group attached to a finite group $H$ is cocompact, whence
	\[
		\Hom_{\Pyk(\Grp)}\left(\lim_{\alpha\in \Lambda^{\op}} G_i^{\disc},H^{\disc}\right) \simeq \colim_{\alpha\in \Lambda}\Hom_{\Grp}(G_i,H) \period
	\]
\end{exm}

\begin{exm}
	The category $\Pyk(\Ab)$ is an abelian category, and 
	the category of compactly generated topological abelian groups embeds fully faithfully into $\Pyk(\Ab)$, in a manner that preserves tiny limits.
	Thus for any abelian group $A$, one obtains a discrete pyknotic abelian group $A^{\disc}$, but for example an infinite product
	\[
		\prod_{a\in I} A_i^{\disc}
	\]
	of finite abelian groups is not discrete.
	To see this explicitly, the discrete abelian group attached to a finite abelian group $B$ is cocompact, whence
	\[
		\Hom_{\Pyk(\Ab)}\left(\prod_{i\in I} A_i^{\disc},B^{\disc}\right) \simeq \bigoplus_{i\in I}\Hom_{\Ab}(A_i,B) \period
	\]
	The limits
	\[
		\widehat{\ZZ} \coloneq \lim_{m\in \NN^{\ast}} (\ZZ/m\ZZ)^{\disc} \andeq \widehat{\ZZ}_p \coloneq \lim_{n\in \NN_0} (\ZZ/p^n\ZZ)^{\disc}
	\]
	are similarly not discrete.
\end{exm}

\begin{exm}
	Let $A$ be a locally compact abelian group.
	Then we can define an abelian variant of our pyknotic set $P_X$:
	for any Stonean space $ K $, form the quotient group
	\[
		P_A(K) \coloneq \Map^{\cts}(K, A)/\Map^{\cts}(K, A^{\disc}) \period
	\]
	The underlying abelian group of $P_A$ is always trivial, but if $A$ is nontrivial, then $P_A$ is as well.
	Thus $A^{\disc} \to A$ is a monomorphism of pyknotic abelian groups, and $P_A$ is the cokernel $A/A^{\disc}$.
\end{exm}

\begin{exm}
	Thanks to \Cref{lem:colimitinPykSofnicetopspaces}, it is not only limits that are preserved by the embedding of compactly generated abelian groups into $\Pyk(\Ab)$.
	For example, let $ E $ be a local field.
	Then since $ E $ is a locally compact topological space, $ E $ is compactly generated.
	The separable closure $ \overline{E} $ is a hausdorff topological space, and $ \overline{E} $ can be obtained as the colimit of a tower
	\begin{equation*}
		\begin{tikzcd}[sep=1em]
			E \arrow[r, hooked] & E_1 \arrow[r, hooked] & E_2 \arrow[r, hooked] & \cdots \comma
		\end{tikzcd}
	\end{equation*}
	where the $ E_n \subset E_{n+1} $ is a finite extension of local fields.
	It follows from \Cref{lem:colimitinPykSofnicetopspaces} that the image of the compactly generated abelian group $ \overline{E} $ in $ \Pyk(\Ab) $ coincides with the the filtered colimit $ \colim_{n} E_n $ in $ \Pyk(\Ab) $.
\end{exm}

\begin{exm}
	Consider the derived category $\DD^-(\Ab)$ of abelian groups, and form the \defn{pyknotic derived category} $\DD^-_{\Pyk}(\Ab) \coloneq \Pyk(\DD^-(\Ab))$, which is a stable category.
	Here, we may compute $\Ext$ groups between pyknotic abelian groups, and we see that they may have cohomological dimension $2$.
	For example, let $\el$ be a prime number, and let $M$ be the cokernel in $\Pyk(\Ab)$ of the inclusion $(\ZZ/\el\ZZ^{\disc})^{\oplus\omega} \inclusion (\ZZ/\el\ZZ^{\disc})^{\times\omega}$.
	Since $\ZZ/\el\ZZ^{\disc}$ is cocompact, and since $\Ext$s of discrete pyknotic abelian groups can be computed in $\Ab$, we find that $\Ext^2_{\DD^-_{\Pyk}(\Ab)}(M,\ZZ/\el\ZZ^{\disc})$ does not vanish.
\end{exm}

This example is the same as the one found at the very end of Hoffmann--Spitzweck \cite{MR2329311};
accordingly, Dustin Clausen and Peter Scholze have proved the following result.

\begin{thm}[{Clausen--Scholze \cite[Corollary 4.9]{Scholze:condensednotes}}]
	Regard the category $ \LCA $ of locally compact abelian groups as a full subcategory of the degree $ 0 $ part of the Pyknotic derived category $ \DD_{\Pyk}^-(\Ab) $.
	Then the induced functor
	\begin{equation*}
		\fromto{\DD^b(\LCA)}{h\DD_{\Pyk}^-(\Ab)}
	\end{equation*}
	is fully faithful; here $ \DD^b(\LCA) $ is the derived category of $ \LCA $ in the sense of Hoffmann--Spitzweck \cite{MR2329311}.
\end{thm}


\subsection{Pyknotic objects of topoi}



\begin{ntn}
	Let $ C $ be a presentable category.
	For each integer $ n \geq -2 $, we write $ C_{\leq n} \subset C $ for the full subcategory spanned by the $ n $-truncated objects, and $ \tau_{\leq n} \colon \fromto{C}{C_{\leq n}} $ for the \defn{$ n $-truncation functor}, left adjoint to the inclusion $ \incto{C_{\leq n}}{C} $.
\end{ntn}

\begin{nul}\label{nul:truninPykofatopos}
	Let $ \XX $ be a topos and $ n \geq -2 $ an integer.
	The $ n $-truncation functor $ \tau_{\leq n} \colon \fromto{\XX}{\XX_{\leq n}} $ preserves finite products \HTT{Lemma}{6.5.1.2}, so we have a natural identification
	\begin{equation*}
		\Pyk(\XX)_{\leq n} = \Pyk(\XX_{\leq n}) \period
	\end{equation*}
	Under this identification, the $ n $-truncation functor $ \tau_{\leq n} \colon \fromto{\Pyk(\XX)}{\Pyk(\XX)_{\leq n}} $ is identified with
	\begin{equation*}
		\Pyk(\tau_{\leq n}) \colon \fromto{\Pyk(\XX)}{\Pyk(\XX_{\leq n})} \period
	\end{equation*} 
\end{nul}

\begin{lem}\label{lem:Pykofhypercomplete}
	Let $ \XX $ be a hypercomplete topos.
	Then the topos $ \Pyk(\XX) $ is hypercomplete.
\end{lem}

\begin{proof}
	We need to show that if $ f \colon \fromto{U}{V} $ is a morphism in $ \Pyk(\XX) $ and for all $ n \geq -2 $ the morphism $ \tau_{\leq n}(f) \colon \fromto{\tau_{\leq n}(U)}{\tau_{\leq n}(V)} $ is an equivalence, then $ f $ is an equivalence.
	In this case, by \Cref{nul:truninPykofatopos} for each complete Boolean algebra $ B $ and integer $ n \geq -2 $, the morphism
	\begin{equation*}
		\tau_{\leq n}(f(B)) \colon \fromto{\tau_{\leq n}(U(B))}{\tau_{\leq n}(V(B))} 
	\end{equation*}
	is an equivalence.
	Since $ \XX $ is hypercomplete, this shows that for all $ B \in \Boolcomp $, the morphism $ f(B) \colon \fromto{U(B)}{V(B)} $ is an equivalence.
	Since equivalences in $ \Pyk(\XX) $ are checked objectwise, this shows that $ f $ is an equivalence.
\end{proof}

\begin{cor}\label{cor:PykSashypercomp}
	Restriction of presheaves $ \fromto{\Sheff(\Comp)}{\Pyk(\Space)} $ induces an equivalence
	\begin{equation*}
		\equivto{\Sheff^{\hyp}(\Comp)}{\Pyk(\Space)} \period
	\end{equation*}
\end{cor}

\begin{cor}\label{cor:PykSenoughpoints}
	The topos $ \Pyk(\Space) $ has enough points.
\end{cor}

\begin{proof}
	Since $ \Pyk(\Space) $ is the hypercompletion of the $ 1 $-localic coherent topos $ \Sheff(\Comp) $ this follows from the higher-categorical Deligne Completeness Theorem \SAG{Theorem}{A.4.0.5} and \SAG{Proposition}{A.2.2.2}.
\end{proof}

\begin{nul}
	Since the terminal object of $ \Pyk(\Space) $ is given by $ \Gammaupperstar(1_{\Space}) $ where $ 1_{\Space} \in \Space $ is the terminal object, the datum of a point of a pyknotic space $ X $ is the datum of a point of the underlying space $ X(\pt) \in \Space $.
	Hence the category $ \Pyk(\Space)_{\ast} $ of pointed objects in $ \Pyk(\Space) $ is canonically identified with the category $ \Pyk(\Space_{\ast}) $ of pyknotic pointed spaces.
\end{nul}

\begin{exm}
	Composition with $\pi_k \colon \Space_{\ast} \to P$ defines a functor
	\[
		\pi_k \colon \Pyk(\Space)_{\ast} \to \Pyk(P)\comma
	\]
	where $P$ is the category
	\[
		\left.
		\begin{aligned}
    		\Set_{\ast} & \\
    		\Grp & \\
    		\Ab
  		\end{aligned}
  		\right\}
		\text{\ when \ }k \kern0.25em
		\left\{
		\begin{aligned}
    		& = 0 \\
    		& = 1 \\
    		& \geq 2 \period
  		\end{aligned}
  		\right.
	\]
	These functors are collectively conservative, so that a morphism $f\colon X \to Y$ of pyknotic spaces is an equivalence if and only if for every $k\geq 0$, the morphism $\pi_k(f)$ is an isomorphism of pointed pyknotic sets, pyknotic groups, or pyknotic abelian groups, as appropriate.

	The upshot here is that pyknotic spaces have pyknotic homotopy groups.
\end{exm}

\begin{nul}
	If $ X \in \Pyk(\Space) $ is a coherent object, then the pyknotic set $ \pi_0(X) = \tau_{\leq 0}(X) $ is a coherent object of the coherent $ 1 $-topos $ \Pyk(\Set) $, hence representable by a compactum.
	More generally, for every point $ x \in X $ and integer $ n \geq 1 $, the homotopy pyknotic group $ \pi_n(X,x) $ is representable by a compact hausdorff group (abelian if $ n \geq 2 $).
\end{nul}

Now we analyze the Postnikov completeness of $ \Pyk(\XX) $.

\begin{ntn}
	For categories $ X $ and $ Y $ with finite products, write 
	\begin{equation*}
		\Funcross(X,Y) \subset \Fun(X,Y)
	\end{equation*}
	for the full subcategory spanned by those functors $ \fromto{X}{Y} $ that preserve finite products.
	Write \smash{$ \Catfp_{\delta_2} \subset \Cat_{\delta_2} $} for the subcategory with objects categories with finite products and morphisms functors that preserve finite products.

	Recall that the forgetful functor $ \fromto{\Catfp_{\delta_2}}{\Cat_{\delta_2}} $ preserves small limits.
	It follows readily that the functor
	\begin{equation*}
		\Funcross(B,-) \colon \fromto{\Catfp_{\delta_2}}{\Catfp_{\delta_2}}
	\end{equation*}		
	preserves small limits as well.
\end{ntn}

\begin{lem}\label{lem:Pykofpostnikovcomplete}
	Let $ \XX $ be a Postnikov complete topos.
	Then the topos $ \Pyk(\XX) $ is Postnikov complete.
\end{lem}

\begin{proof}
	Since $ \XX $ is Postnikov complete, the natural functor
	\begin{equation*}
		\fromto{\XX}{\textstyle\lim_{n} \XX_{\leq n}}
	\end{equation*}
	to the inverse limit in $ \Cat $ along the $ n $-truncation functors $ \tau_{\leq n} \colon \fromto{\XX_{\leq n+1}}{\XX_{\leq n}} $ is an equivalence \SAG{Theorem}{A.7.2.4}.
	Since the $ n $-truncation functors on a topos preserve finite products \HTT{Lemma}{6.5.1.2}, we obtain an equivalence
	\begin{equation}\label{eq:Pykoftower}
		\equivto{\Pyk(\XX)}{\textstyle \lim_{n} \Pyk(\XX_{\leq n})} \comma
	\end{equation}
	where the latter inverse limit is computed in $ \Cat_{\delta_2} $ along the functors 
	\begin{equation*}
		\Pyk(\tau_{\leq n}) \colon \fromto{\Pyk(\XX_{\leq n+1})}{\Pyk(\XX_{\leq n})} \period
	\end{equation*}
	In light \Cref{nul:truninPykofatopos}, the equivalence \eqref{eq:Pykoftower} shows that $ \Pyk(\XX) $ is Postnikov complete.
\end{proof}

\begin{exm}
	In particular, $\Pyk(\Space)$ is Postnikov complete.
	Hence any pyknotic space $X$ can be exhibited as the limit of its Postnikov tower
	\[
		X \to \cdots \to \tau_{\leq 2} X \to \tau_{\leq 1} X \to \tau_{\leq 0} X \to \tau_{\leq -1} X \to \tau_{\leq -2} X = \pt \comma
	\]
	and the fibre of $\tau_{\leq k} X \to \tau_{\leq k-1} X$ over a point is $k$-truncated and $k$-connected.
	Since $\Pyk(\Space)$ has homotopy dimension $0$ (\Cref{nul:PykSislocal}), it follows that each of these fibres is the classifying pyknotic space $B^k\pi_k(X)$, where $\pi_k(X)$ is:
	\[
		\left.
		\begin{aligned}
    		\text{either empty or $\pt$} & \\
    		\text{a pointed pyknotic set} & \\
    		\text{a pyknotic group} & \\
    		\text{a pyknotic abelian group}
  		\end{aligned}
  		\right\}
		\text{\ when \ }k \kern0.25em
		\left\{
		\begin{aligned}
    		& = -1 \\
    		& = 0 \\
    		& = 1 \\
    		& \geq 2 \phantom{-} \period
  		\end{aligned}
  		\right.
	\]
\end{exm}


\subsection{Tensor products of pyknotic objects}

Let $D^{\otimes}$ be a presentably symmetric monoidal category -- i.e., a presentable category with a symmetric monoidal structure in which the tensor product functor $D \times D \to D$ preserves colimits separately in each variable.
Let $X$ and $Y$ be two pyknotic objects of $D$;
we now set about showing that their tensor product $X \otimes_D Y$ admits a canonical pyknotic structure.

\begin{cnstr}
	Let $D^{\otimes}$ be a presentably symmetric monoidal category.
	Thus $D^{\otimes}$ is a commutative algebra object in $\Pr^L$.

	Since $\Comp$ is a symmetric monoidal category under the product, the Day convolution symmetric monoidal structure on $\Fun(\Comp^{\op}, D)$ coincides with the objectwise tensor product.
	The localisation functor $\Fun(\Comp^{\op}, D) \to \Pyk(D)$ is compatible with this symmetric monoidal structure, and so we obtain a symmetric monoidal structure $\Pyk(D)^{\otimes}$ on $\Pyk(D)$.

	Equivalently, the product of pyknotic spaces preserves colimits separately in each variable, so we obtain a presentably symmetric monoidal category $\Pyk(\Space)^{\times}$.
	Now we can identify
	\[
		\Pyk(D)^{\otimes} \simeq \Pyk(\Space)^{\times} \otimes D^{\otimes} \comma
	\]
	the tensor product (=coproduct) of the commutative algebras in $\Pr^L$.

	To be explicit, if $X$ and $Y$ are pyknotic objects of $D$, then their tensor product is the pyknotic object $X \otimes_{\Pyk(D)} Y$ that is the hypersheafification of the assignment
	\[
		K \mapsto X(K) \otimes_D Y(K) \period
	\]
	The unit is the discrete pyknotic object attached to the unit of $D$.
\end{cnstr}

\begin{exm}
	If the presentably symmetric monoidal category $D^{\otimes}$ is cartesian, then so is the symmetric monoidal structure $\Pyk(D)^{\otimes}$.
\end{exm}

\begin{nul}\label{nul:Gammalowerstarsymmetricmonoidal}
	Let $D^{\otimes}$ be presentably symmetric monoidal.
	Then by construction, the discrete functor $D \to \Pyk(D)$ extends to a symmetric monoidal left adjoint $D^{\otimes} \to \Pyk(D)^{\otimes}$, so that for any objects $U$ and $V$ of $D$, we have a natural equivalence
	\[
		U^{\disc} \otimes_{\Pyk(D)} V^{\disc} \simeq (U \otimes_D V)^{\disc} \period
	\]
	Since $\Gammalowerstar \colon \Pyk(\Space) \to \Space$ preserves finite products, it is also naturally symmetric monoidal, whence the functor $\Gammalowerstar \colon \Pyk(D) \to D$ is symmetric monoidal as well, so that for any two pyknotic objects $X$ and $Y$ of $D$, we obtain an equivalence
	\[
		X \otimes_D Y \simeq X \otimes_{\Pyk(D)} Y \period
	\]
	Also, if $X$ is an object of $D$ and if $Y$ is a pyknotic object of $D$, then there are equivalences in $D$
	\[
		\MOR_{\Pyk(D)}(X^{\disc}, Y) \simeq \MOR_D(X, Y) \andeq  \MOR_{\Pyk(D)}(Y, X^{\indisc}) \simeq \MOR_D(Y, X) \period
	\]
\end{nul}

\begin{exm}
	Let $A$ and $B$ be two pyknotic abelian groups.
	Then their tensor product $A \otimes B$ admits a canonical pyknotic structure.
	For example, one can form the adèles of $\QQ$ as a pyknotic abelian group in this manner:
	\[
		\AA_{\QQ} \coloneq (\widehat{\ZZ} \times \RR) \otimes_{\Pyk(\Ab)} \QQ^{\disc}\period
	\]
\end{exm}


\subsection{\texorpdfstring{$\Pyk$}{Pyk}-modules}

A \textit{$\Pyk$-module} is a presentable category $ C $ along with a functor
\[
	\Comp \times C \to C \comma \qquad (K, X) \mapsto K \otimes X
\]
equipped with equivalences $\pt \otimes X \simeq X$ and $(K \times L) \otimes X \simeq K \otimes (L \otimes X)$, which plays the rôle of a `continuous coproduct' of $X$ with itself indexed over the points of $K$.
Accordingly, we will insist upon the following axioms.
\begin{itemize}
	\item For any compactum $K$ and any small diagram $X \colon I \to C$, the natural map
	\[
		\textstyle\colim_{i\in I}(K \otimes X_i) \to K \otimes (\textstyle\colim_{i\in I} X_i)
	\]
	is an equivalence.

	\item For any object $X$ of $ C $ and any two compacta $K$ and $L$, the natural map
	\[
		(K \otimes X) \sqcup (L \otimes X) \to (K \sqcup L) \otimes X
	\]
	is an equivalence.

	\item For any object $X \in C$, any compactum $K$, and any hypercover $\surjto{L_{\ast}}{K}$, the natural map
	\[
		\colim\nolimits_{\mbfDelta^{\op}} L_{\ast} \otimes X \to K \otimes X
	\]
	is an equivalence.
\end{itemize}
This can all be expressed compactly (and with full homotopy coherence) in the following.

\begin{dfn}
	A \defn{$\Pyk$-module} is a module over the commutative algebra $\Pyk(\Space)$ in $\Pr^L$.
	A \defn{commutative $\Pyk$-algebra} is an object under $\Pyk(\Space)^{\times}$ in $\CAlg(\Pr^{L,\otimes})$.
\end{dfn}

\begin{exm}
	If $D$ is a presentable category, then $\Pyk(D)$ is a $\Pyk$-module, and if $D^{\otimes}$ is a presentably symmetric monoidal category, then $\Pyk(D)^{\otimes}$ is a $\Pyk$-algebra.
\end{exm}

\begin{nul}
	A $\Pyk$-module structure on a presentable category $ C $ is thus a left adjoint functor $\alpha^{\ast} \colon \Pyk(C) \to C$ along with equivalences
	\[
		\alpha^{\ast} \Gammaupperstar \simeq \id_{C} \andeq \alpha^{\ast}\Delta^{\ast} \simeq \alpha^{\ast}\Pyk(\alpha^{\ast})
	\]
	(and their higher-order analogues), where $\Delta^{\ast} \colon \Pyk(\Pyk(C)) \to \Pyk(C)$ is the pullback along the diagonal $\Comp \to \Comp \times \Comp$.
\end{nul}

Thus a $\Pyk$-module can also be specified by a presentable category $ C $ along with a functor
\[
	C \times \Comp^{\op} \to C \comma \qquad (X, K) \mapsto X^K \comma
\]
along with equivalences $X^\pt \simeq X$ and $X^{(K \times L)} \simeq (X^K)^L$, which plays the rôle of a `continuous product' of $X$ with itself indexed over the points of $K$  subject to the following axioms.
\begin{itemize}
	\item For any compactum $K$ and any small diagram $X \colon I \to C$, the natural map
	\[
		(\textstyle\lim_{i\in I} X_i)^K \to \lim_{i\in I} (X_i^K)
	\]
	is an equivalence.

	\item For any object $X$ of $ C $ and any two compacta $K$ and $L$, the natural map
	\[
		X^{(K \sqcup L)} \to X^K \times X^L
	\]
	is an equivalence.

	\item For any object $X \in C$, any compactum $K$, and any hypercover $\surjto{L_{\ast}}{K}$, the natural map
	\[
		X^K \to \lim\nolimits_{\mbfDelta} X^{L_{\ast}}
	\]
	is an equivalence.
\end{itemize}

\begin{nul}
	Note that if $ C $ is a $\Pyk$-module, then for any object $X$ of $ C $ and any compactum $K$, we obtain morphisms
	\[
		 \coprod_{k\in|K|} X \to K \otimes X \andeq X^K \to \prod_{k\in |K|} X \simeq X^{\indisc}(K) \comma
	\]
	natural in both $X$ and $K$.
	These morphisms are generally not equivalences.
	For example, there exists a small regular cardinal $\kappa$ such that $\alpha^{\ast} \colon \Pyk(C) \to C$ carries $\kappa$-compact objects to $\kappa$-compact objects.
	Thus if $X$ is $\kappa$-compact, so is $K \otimes X$, for \emph{any} compactum $K$;
	this will generally not be true of the coproduct $\coprod_{k\in|K|} X$.
\end{nul}

\begin{nul}
	For any presentable category $ C $, the category $\Pyk(C)$ is the free $\Pyk$-module generated by $ C $.
\end{nul}


\section{Pyknotic objects in algebra \& homotopy theory}


\subsection{Pyknotic spectra \& pyknotic homotopy groups}

In this subsection we investigate the category $ \Pyk(\Sp) $ of pyknotic spectra.
It is a formal matter to see that this agrees with the stabilisation of the category of pyknotic spaces.

\begin{ntn}
	Let $ C $ be a category with pushouts and a terminal object and let $ D $ be a category with finite limits.
	We write
	\begin{equation*}
		\Exc_{*}(C,D) \subset \Fun(C,D)
	\end{equation*}  
	for the full subcategory spanned by the \textit{reduced excisive functors} \HA{Definition}{1.4.2.1}.
\end{ntn}

\begin{nul}
	Let $ B $ be a category with finite products and $ D $ a category with finite limits.
	Then $ \Funcross(B,D) $ admits finite limits, which are computed pointwise.
\end{nul}

We'll record a few facts for future use.
All are immediate from the definitions.

\begin{lem}\label{lem:prodFuncommuteswithFun}
	Let $ B $, $ C $, and $ D $ be categories, and assume that $ B $ and $ D $ have finite products.
	Then the natural equivalence of categories
	\begin{equation*}
		\Fun(B,\Fun(C,D)) \equivalent \Fun(C,\Fun(B,D))
	\end{equation*}
	restricts to an equivalence
	\begin{equation}\label{eq:equivprodpreserve}
		\Funcross(B,\Fun(C,D)) \equivalent \Fun(C,\Fun^{\times}(B,D)) \period
	\end{equation}
\end{lem}

\begin{exm}
	Let $ C $ and $ D $ be categories, and assume that $ D $ has finite products.
	Then we have a natural equivalence of categories
	\begin{equation*}
		\Pyk(\Fun(C,D)) \equivalent \Fun(C,\Pyk(D)) \period
	\end{equation*}
\end{exm}

\begin{lem}\label{lem:PykandSpcommute}
	Let $ B $, $ C $, and $ D $ be categories.
	Assume that $ B $ has finite products, $ C $ has pushouts and a terminal object, and $ D $ has finite limits.
	Then the natural equivalence of categories \eqref{eq:equivprodpreserve} restricts to an equivalence
	\begin{equation*}
		\Funcross(B,\Exc_{*}(C,D)) \equivalent \Exc_{*}(C,\Fun^{\times}(B,D)) \period
	\end{equation*}
\end{lem}

\begin{exm}
	Taking $ B = \Boolcomp $ and $ C $ to be the category $ \Space_{*}^{\fin} $ of finite pointed spaces in \Cref{lem:PykandSpcommute} we deduce that we have an equivalence
	\begin{equation*}
		\Pyk(\Sp(D)) \equivalent \Sp(\Pyk(D))
	\end{equation*}
	natural in categories $ D $ with finite limits (cf. \HA{Definition}{1.4.2.8}).
\end{exm}

\begin{exm}
	\Cref{lem:Pykofpostnikovcomplete} shows that $ \Pyk(\Sp) $ is the stabilisation of a Postnikov complete topos.
\end{exm}

\begin{nul}
	If $ D $ is a $ 1 $-category with finite products, then we have a nautral equivalence of $ 1 $-categories
	\begin{equation*}
		\Pyk(\Ab(D)) \equivalent \Ab(\Pyk(D))
	\end{equation*}
	between pyknotic objects in the category $ \Ab(D) $ of ableian group objects in $ D $ and abelian group objects in $ \Pyk(D) $.
\end{nul}

\begin{ntn}
	For a topos $ \XX $, write $ \PykSp(\XX) \coloneq \Pyk(\Sp(\XX)) $.
	Write
	\begin{equation*}
		 \PykSp_{\geq 0}(\XX) \subset \PykSp(\XX) \andeq \PykSp_{\leq 0}(\XX) \subset \PykSp(\XX)
	\end{equation*}
	for the full subcategories spanned by the connective and coconnective objects, respectively.
\end{ntn}

\begin{prp}\label{prop:tstruncture}
	Let $ \XX $ be a topos.
	Then:
	\begin{enumerate}[{\upshape (\ref*{prop:tstruncture}.1)}]
		\item\label{prop:tstruncture.1} The full subcategories $ (\PykSp_{\geq 0}(\XX),\PykSp_{\leq 0}(\XX)) $ deterine an accessible $ t $-structure on $ \PykSp(\XX) $.

		\item\label{prop:tstruncture.2} The full subcategory $ \PykSp_{\leq 0}(\XX) \subset \PykSp(\XX) $ is closed under filtered colimits.

		\item\label{prop:tstruncture.3} The $ t $-structure on $ \PykSp(\XX) $ is right complete.

		\item\label{prop:tstruncture.4} The functor $ \pi_0 \colon \fromto{\PykSp(\XX)}{\Pyk(\XX)_{\leq 0}} $ determines an equivalence of categories
		\begin{equation*}
			\equivto{\PykSp(\XX)^{\heart}}{\Pyk(\Ab(\XX_{\leq 0}))} \period
		\end{equation*}

		\item\label{prop:tstruncture.5} If, in addition, $ \XX $ is Postnikov complete, then the $ t $-structure on $ \PykSp(\XX) $ is left complete.
	\end{enumerate}
\end{prp}

\begin{proof}
	Items \enumref{prop:tstruncture}{1}--\enumref{prop:tstruncture}{4} follow from \cites[Proposition 1.7]{DAGVII} (see also \SAG{Proposition}{C.5.2.8}).
	\Cref{lem:Pykofpostnikovcomplete} and \cites[Warning 1.8]{DAGVII} imply \enumref{prop:tstruncture}{5}.
\end{proof}

\begin{exm}
	The $ t $-structure on $ \Pyk(\Sp) $ is both left and right complete and the heart $ \Pyk(\Sp)^{\heart} $ is canonically equivalent to the category $ \Pyk(\Ab) $ of pyknotic abelian groups.
	Consequently, the homotopy groups of a pyknotic spectrum are pyknotic abelian groups.
	
	Moreover, since stabilisation is functorial in categories with finite limits and left exact functors, from \Cref{nul:PykSislocal} we get a chain of adjoints
	\begin{equation*}
		\begin{tikzcd}[sep=5.5em]
			\Pyk(\Sp) \arrow[r,"\Gammalowerstar"{xshift=0ex,description}] & \Sp \period \arrow[l, hooked', shift left=1.5ex, "\Gammauppershriek"{xshift=2.5ex,description}] \arrow[l, hooked', shift right=1.5ex, "\Gammaupperstar"{xshift=-2.5ex,description}]
		\end{tikzcd}
	\end{equation*}
	From \cite[Remark 1.9]{DAGVII} we deduce that the functors $ \Gammaupperstar \colon \incto{\Sp}{\Pyk(\Sp)} $ and $ \Gammalowerstar \colon \fromto{\Pyk(\Sp)}{\Sp} $ are $ t $-exact, and the functor $ \Gammauppershriek \colon \incto{\Sp}{\Pyk(\Sp)} $ is left $ t $-exact.
	Also note that the square of right adjoints 
	\begin{equation*}
		\begin{tikzcd}[sep=2em]
			\Pyk(\Sp) \arrow[d,"\Gammalowerstar"'] \arrow[r, "\Omega^{\infty}"] & \Pyk(\Space) \arrow[d,"\Gammalowerstar"] \\ 
			\Sp \arrow[r, "\Omega^{\infty}"'] & \Space 
		\end{tikzcd}
	\end{equation*}
	commutes.
\end{exm}

\begin{exm}
	If $A$ is a pyknotic abelian group, then we also write $A$ for the pyknotic spectrum obtained by composing $A$ with the Eilenberg--Mac Lane functor $\Ab \to \Sp$.
\end{exm}

\begin{nul}
	Stabilising the embedding of profinite spaces into pyknotic spaces (\Cref{exm:profinitespacesembedintoPykS}) we obtain an embedding
	\begin{equation*}
		\incto{\Sp(\Pro(\Spacefin))}{\Pyk(\Sp)} \period
	\end{equation*}
\end{nul}

\begin{nul}
	Let $ C $ be a presentable category.
	By the universal property of the category of proöbjects in $ C $, the discrete functor $ \Gammaupperstar \colon \fromto{C}{\Pyk(C)} $ extneds to a functor $ \fromto{\Pro(C)}{\Pyk(C)} $, which admits a left adjoint $ \Gamma_{!} \colon \fromto{\Pyk(C)}{\Pro(C)} $.
	The materialisation functor $ \mat \colon \fromto{\Pro(C)}{C} $ \SAG{Example}{A.8.1.7} then factors as the composite
	\begin{equation*}
		\begin{tikzcd}[sep=1.5em]
			\Pro(C) \arrow[r] & \Pyk(C) \arrow[r, "\Gammalowerstar"] & C \period
		\end{tikzcd}
	\end{equation*}
\end{nul}

\begin{exm}\label{exm:Enilpotentcomp}
	Let $ E $ be an $ E_1 $-ring spectrum.
	Write $ E^{\tensor \ast} $ for the \emph{Amitsur complex} -- the augmented cosimiplicial spectrum
	\begin{equation*}
		\augcosimplicial{S^0}{E}{E^{\tensor 2}}{E^{\tensor 3}} \period
	\end{equation*}
	For a spectrum $ X $, the \defn{$ E $-nilpotent completion} $ X_{E}^{\wedge} $ is the limit of the $ \Tot $-tower
	\begin{align*}
		X_{E}^{\wedge} &\coloneq \textstyle\lim_{n} \Tot^n(X \tensor_{\Sp} E^{\tensor \ast}) \\
		&\equivalent \textstyle\lim_{\mbfDelta} X \tensor_{\Sp} E^{\tensor \ast} \period
	\end{align*}
	See \cites[\S5]{Hopkins:complexoriented}{Lurie:Chromatic30}[\S2.1]{MR3570153}.
	Regarding the $ \Tot $-tower as a prospectrum and applying the functor $ \fromto{\Pro(\Sp)}{\Pyk(\Sp)} $, we obtain the \defn{pyknotic $E$-nilpotent completion}
	\begin{equation*}
		X_{E}^{\wedge,\pyk} \coloneq \textstyle \lim_n \Tot^n(X \tensor_{\Sp} E^{\tensor \ast})^{\disc} \comma
	\end{equation*}
	which has underlying spectrum the usual $ E $-nilpotent completion $ X_{E}^{\wedge} $.
	Since $ \Gammaupperstar $ does not preserve limits in general, the pyknotic $ E $-nilpotent completion \smash{$ X_{E}^{\wedge,\pyk} $} is generally not discrete.
	Rather, the pyknotic $E$-nilpotent completion is a pyknotic refinment of the $ E $-nilpotent completion $ X_{E}^{\wedge} $.

	Note also that since $ \Gammaupperstar $ preserves finite limits and is symmetric monoidal \Cref{nul:Gammalowerstarsymmetricmonoidal}, we can describe \smash{$ X_{E}^{\wedge,\pyk} $} as the limit
	\begin{equation*}
		X_{E}^{\wedge,\pyk} \equivalent \textstyle \lim_n \Tot^n(X^{\disc} \tensor_{\Pyk(\Sp)} (E^{\disc})^{\tensor \ast}) \period
	\end{equation*}
	Thus the pyknotic $E$-nilpotent completion is the result of forming the $E^{\disc}$-nilpotent completion of $X^{\disc}$ in pyknotic spectra
\end{exm}





\subsection{Pyknotic rings and pyknotic modules}

\begin{dfn}
	A \defn{pyknotic ring} is nothing more than a pyknotic object in the category of rings (which we will usually assume are commutative).
	A \defn{pyknotic module} over a pyknotic ring $R$ is an $R$-module in $\Pyk(\Ab)$.
\end{dfn}

\begin{exm}
	Any normed ring is compactly generated, and so they are pyknotic rings.
	In particular, $\ZZ$, $\QQ$, $\RR$, $\CC$, any local field $E$, any algebraic closure thereof, $\CC_p$, all Banach rings, \textit{\& c.}, are all pyknotic rings in a natural manner.
\end{exm}

\begin{exm}
	For any global field $K$, the adèle group $\AA_K$ is a locally compact hausdorff ring, whence it is a compactly generated ring, whence it is a pyknotic ring.
	More generally, if $S$ is a set, and $\{i_s \colon B_s \to A_s\}_{s \in S}$ is a family of pyknotic ring homomorphisms, then the \defn{restricted product} is the pyknotic ring 
	\[
		\restrictedprod_{s\in S} A_s \coloneq \colim_{W \in P^{\fin}(S)} \prod_{w\in W} A_w \times \prod_{w\in S\smallsetminus W} B_w \comma
	\]
	where $P^{\fin}(S)$ is the poset of finite subsets of $S$.
\end{exm}

\begin{exm}
	Over any normed ring $R$, any first countable (and thus metrisable) topological $R$-module admits a natural pyknotic structure.
\end{exm}

\begin{cnstr}
	Let $A$ be an associative pyknotic ring.
	For example, $A$ may be a topological ring with a compactly generated topology.
	Viewed as a pyknotic spectrum, $A$ has the natural structure of a \defn{pyknotic $E_1$ ring} -- i.e., an $E_1$ algebra in $\Pyk(\Sp)^{\otimes}$.
	If $A$ is commutative, then $A$ is $E_{\infty}$.

	We may therefore define the \defn{pyknotic derived category} $\DD_{\Pyk}(A)$ as the category of left $A$-modules in $\Pyk(\Sp)$.
\end{cnstr}


\subsection{The proétale topos as a \texorpdfstring{$\Pyk$}{Pyk}-algebra}

\begin{ntn}
	For a topos $ \XX $, we write $ \XXcohbdd \subset \XX $ for the full subcategory spanned by the \defn{truncated coherent} objects -- those objects that are both truncated in coherent.
	Recall that if $ \XX $ is a coherent topos, then $ \XXcohbdd $ is a bounded pretopos \SAG{Example}{A.7.4.4}.
\end{ntn}

\begin{cnstr}
	Let $\XX$ be a bounded coherent topos, and let
	\[
		C \coloneq \XX^{\coh}_{<\infty} \subset \XX
	\]
	be the bounded pretopos of truncated coherent objects of $\XX$.
	Form the (small) category $\Pro^{\delta_0}(C)$ of proöbjects relative to $\delta_0$ of $ C $.
	This is the universal category with all tiny inverse limits generated by $ C $.
	The category $ \Pro^{\delta_0}(C) $ is not a pretopos, but the collection $\eff$ of effective epimorphisms endows it with the structure of a presite (\Cref{dfn:presite}).
	Consequently, we may form the hypercomplete, coherent, and locally coherent topos
	\[
		\XX^{\dag} \coloneq \Sheff^{\hyp}(\Pro^{\delta_0}(C))
	\]
	\cite[Propositions \SAGthmlink{A.2.2.2} \& \SAGthmlink{A.3.1.3}]{SAG}.
	We call $\XX^{\dag}$ the \defn{solidification} of the topos $\XX$.
	
\end{cnstr}

\begin{nul}
	By \cite[Corollary 2.8]{Haine:1-localic}, if $ \flowerstar \colon \fromto{\XX}{\YY} $ is a coherent geometric morphism of bounded coherent topoi, then the induced geometric morphism $ \flowerstar^{\dag} \colon \fromto{\XX^{\dag}}{\YY^{\dag}} $ is coherent.
\end{nul}

\begin{exm}\label{exm:effepionProsubcanon}
	From \Cref{cnstr:presitefinitary} it follows that if $ \XX $ is a bounded coherent topos then the effective epimorphism topology on $ \Pro(\XXcohbdd) $ is subcanonical.
	Moreover, since the Yoneda embedding
	\begin{equation*}
		\incto{\Pro(\XXcohbdd)}{\Sheff(\Pro(\XXcohbdd))}
	\end{equation*}
	preserves tiny limits, truncated objects of a topos are hypercomplete, and hypercomplete objects are closed under limits, the Yoneda embedding factors through $ \XX^{\dag} $.
\end{exm}

Our next goal is to show that if $ \XX $ is an \textit{$ n $-localic} coherent topos, then $ \XX^{\dag} $ can be written as hypersheaves on $ \Pro(\XXcoh_{\leq n-1}) $ (\Cref{cor:proembeds}).
To see this, we first need to show that every object of $ \Pro(\XXcohbdd) $ admits an effective epimorphism from an object of $ \Pro(\XXcoh_{\leq n-1}) $.
This requires a number of preliminaries.

\begin{nul}\label{nul:n-truncaedcohbasis}
	Let $ n \geq 1 $ be an integer and let $ \XX $ be a coherent $ n $-localic topos.
	Then by the classification theorem for bounded coherent topoi \SAG{Theorem}{A.7.5.3}, since $ \XX $ is $ n $-localic we have $ \XX \equivalent \Sheff(\XXcoh_{\leq n-1}) $.\footnote{Here we'll only actually use the case $ n = 1 $, which is the content of \cite[2.13]{Haine:1-localic}.}
	Thus $ \XXcoh_{\leq n-1} \subset \XX $ generates $ \XX $ under colimits.
	In particular, for every quasicompact object $ U \in \XX $, there exists an effective epimorphism $ \surjto{\coprod_{i\in I} U_{i}}{U} $ where $ U_i \in \XXcoh_{\leq n-1} $ for each $ i \in I $.
	Since $ U $ is quasicompact, there exists a finite subset $ I_0 \subset I $ such that $ \surjto{\coprod_{i\in I_0} U_{i}}{U} $ is an effective epimorphism.
	Since $ U_i \in \XXcoh_{\leq n-1} $ for each $ i \in I $, we deduce that the finite coproduct $ \coprod_{i\in I_0} U_{i} $ is $ (n-1) $-truncated \HTT{Lemma}{6.4.4.4} and coherent.
	Thus every quasicompact object of $ \XX $ admits an effective epimorphism from a $ (n-1) $-truncated coherent object of $ \XX $.
\end{nul}

Since we must contend with proöbjects, it isn't immediate from \Cref{nul:n-truncaedcohbasis} that every object of $ \Pro(\XXcohbdd) $ admits an effective epimorphism from an object of $ \Pro(\XXcoh_{\leq n-1}) $.
To show this, we'll use the fact that we can always arrange to index a proöbject by a particularly nice poset:

\begin{lem}[{\SAG{Lemma}{E.1.6.4}}]\label{lem:reducetoresidfinite}
	Let $ A' $ be a filtered poset.
	Then there exists a cofinal map of posets $ f \colon \fromto{A}{A'} $, where $ A $ is a filtered poset with the following property:
	\begin{enumerate}
		\item[{\upshape($ \ast $)}]\label{lem:residuallyfinite} For every element $ \alpha \in A $, the set $ \{ \beta \in A \,|\, \beta \leq  \alpha \} $ is finite.
	\end{enumerate}
\end{lem}

\begin{cnstr}
	Let us call a poset $ A $ satisfying (\hyperref[lem:residuallyfinite]{$ \ast $}) \defn{residually finite}.
	If $ A $ is a residually finite poset, then there exists a map of posets $ \rank \colon \fromto{A}{\NN} $ called the \defn{rank} which is determined by the following requirement: $ \rk(\alpha) $ is the smallest natural number not equal to $ \rk(\beta) $ for $ \beta < \alpha $ (cf. \HA{Remark}{A.5.17}).
	In particular, $ \rk(\alpha) = 0 $ if and only if $ \alpha $ is a minimal element of $ A $.
\end{cnstr}

\begin{prp}\label{prp:probasis}
	Let $ n \geq 1 $ be an integer and let $ \XX $ be an $ n $-localic coherent topos.
	Then for every object $ X \in \Pro(\XXcohbdd) $, there exists an effective epimorphism $ \surjto{Y}{X} $ where $ Y \in \Pro(\XXcoh_{\leq n-1}) $.
\end{prp}

\begin{proof}
	Write $ C \coloneq \XXcohbdd $ and $ D \coloneq \XXcoh_{\leq n-1} $.
	Let $ \{X_\alpha\}_{\alpha \in A^{\op}} $ be an object of $ \Pro(C) $, where we without loss of generality assume that $ A $ is a residually finite filtered poset (\Cref{lem:reducetoresidfinite}).
	We construct a morphism $ e \colon \fromto{\{Y_\alpha\}_{\alpha \in A^{\op}}}{\{X_\alpha\}_{\alpha \in A^{\op}}} $ in $ \Pro(C) $ where for each $ \alpha \in A $, each $ e_\alpha \colon \fromto{Y_\alpha}{X_\alpha} $ is an effective epimorphism and $ Y_{\alpha} \in D $.
	We construct this inductively on the rank of elements of $ A $.
	For each $ n \in \NN $, write 
	\begin{equation*}
		A_{\leq  n} \coloneq \{ \alpha \in A \, |\, \rk(\alpha) \leq  n \} \period
	\end{equation*}

	First, for each element $ \alpha \in A $ with $ \rk(\alpha) = 0 $ (i.e., minimal element of $ A $), appealing to \Cref{nul:n-truncaedcohbasis}, choose an effective epimorphism $ e_\alpha \colon \surjto{Y_{\alpha}}{X_{\alpha}} $ where $ Y_\alpha \in D $.
	
	For the induction step, suppose that we have defined a functor $ Y \colon \fromto{A_{\leq  n}^{\op}}{D} $ along with a natural effective epimorphism $ e \colon \surjto{Y}{X|_{A_{\leq  n}}} $; we now extend $ Y $ to $ A_{\leq  n+1} $ as follows.
	For each $ \alpha \in A $ with $ \rank(\alpha) = n+1 $, consider the pulled-back effective epimorphism
	\begin{equation*}
		\surjto{\coprod_{\substack{\beta < \alpha \\ \rank(\beta) = n}} X_{\alpha} \cross_{X_\beta} Y_{\beta}}{X_{\alpha}} \period
	\end{equation*}
	For each $ \beta < \alpha $ with $ \rk(\beta) = n $, appealing to \Cref{nul:n-truncaedcohbasis} we choose an effective epimorphism $ e'_\beta \colon \surjto{Y'_{\beta}}{X_{\alpha} \cross_{X_{\beta}} Y_{\beta}} $, and define the effective epimorphism $ e_{\alpha} \colon \surjto{Y_\alpha}{X_{\alpha}} $ as the composite
	\begin{equation*}
		\begin{tikzcd}
			e_\alpha \colon Y_{\alpha} \coloneq \displaystyle \coprod_{\substack{\beta < \alpha \\ \rank(\beta) = n}} Y'_\beta \arrow[r, ->>, "\coprod_\beta e'_\beta"] & \displaystyle \coprod_{\substack{\beta < \alpha \\ \rank(\beta) = n}} X_{\alpha} \cross_{X_\beta} Y_{\beta} \arrow[r, ->>] & X_\alpha \period
		\end{tikzcd}
	\end{equation*}
	Then by construction the functor $ Y \colon \fromto{A_{\leq  n}^{\op}}{D} $ extends to a functor $ Y \colon \fromto{A_{\leq  n+1}^{\op}}{D} $ equipped with a natural effective epimorphism $ e \colon \surjto{Y}{X|_{A_{\leq  n+1}}} $, as desired.
\end{proof} 

We now prove the desired result using a slight variant of \SAG{Proposition}{A.3.4.2}.

\begin{prp}\label{cor:proembeds}
	Let $ n \geq 1 $ be an integer and let $ \XX $ be an $ n $-localic coherent topos.
	Then restriction of presheaves defines an equivalence
	\begin{equation*}
		\XX^{\dag} \equivalence \Sheff^{\hyp}(\Pro(\XXcoh_{\leq n-1})) 
	\end{equation*}
	with inverse given by right Kan extension.
\end{prp}

\begin{proof}
	Let $ \iupperstar \colon \incto{\Pro(\XXcoh_{\leq n-1})}{\Pro(\XXcoh_{<\infty})} $ denote the inclusion.
	Since $ \Pro(\XXcoh_{\leq n-1}) $ is closed under finite coproducts and finite limits in $ \Pro(\XXcoh_{<\infty}) $, the inclusion $ \iupperstar $ induces a geometric morphism 
	\begin{equation*}
		\ilowerstar \colon \fromto{\XX^{\dag}}{\Sheff^{\hyp}(\Pro(\XXcoh_{\leq n-1}))} \comma 
	\end{equation*} 
	where the right adjoint $ \ilowerstar $ is given by restriction of presheaves \SAG{Proposition}{A.3.3.1}.
	Combining \Cref{prp:probasis} with \cite[\SAGthm{Proposition}{20.4.5.1} \& \SAGthm{Remark}{20.4.5.2}]{SAG} and the hypercompleteness of $ \XX^{\dag} $, we deduce that $ \ilowerstar $ is fully faithful. 
	To complete the proof, it suffices to show that $ \iupperstar $ is fully faithful.
	We do this by showing that $ \ilowerstar $ admits a fully faithful right adjoint $ \iuppershriek $ given by right Kan extension.
	
	For simplicity, we write $ C \coloneq \Pro(\XXcoh_{<\infty}) $ and $ D \coloneq \Pro(\XXcoh_{\leq n-1}) $.
	Let $ F \colon \fromto{D^{\op}}{\Space} $ be a sheaf for the effective epimorphism topology, and let $ \iuppershriek(F) \colon \fromto{C^{\op}}{\Space} $ denote the right Kan extension of $ F $ along the inclusion $ D^{\op} \subset C^{\op} $.
	We claim that $ \iuppershriek(F) $ is a sheaf on $ C $ for the effective epimorphism topology.
	To see this, fix a covering sieve $ S \subset C_{/X} $.
	Set $ D_{/X} \coloneq D \cross_{C} C_{/X} $ and $ T \coloneq D \cross_{C} S $.
	We wish to show that the upper horizontal map in the square
	\begin{equation*}
		\begin{tikzcd}
			\iuppershriek(F)(X) \arrow[r] \arrow[d] & \displaystyle\lim_{X' \in S^{\op}} \iuppershriek(F)(X') \arrow[d] \\
			\displaystyle\lim_{Y \in D_{/X}^{\op}} \iuppershriek(F)(Y) \arrow[r] & \displaystyle\lim_{Y \in T^{\op}} \iuppershriek(F)(Y)
		\end{tikzcd}
	\end{equation*}
	is an equivalence.
	The vertical maps are equivalences because $ \iuppershriek(F) $ is the right Kan extension of $ F $.
	The lower horizontal map is an equivalence because $ F $ is a sheaf on $ D $ and every object of $ C $ is covered by an object of $ D $ (\Cref{prp:probasis}). 
	Thus right Kan extension of presheaves restricts to a fully faithful functor
	\begin{equation*}
		\iuppershriek \colon \incto{\Sheff(\Pro(\XXcoh_{\leq n-1}))}{\Sheff(\Pro(\XXcoh_{<\infty}))}
	\end{equation*}
	which is right adjoint to restriction of presheaves.
	Since the image of a hypercomplete sheaf under the pushforward in a geometric morphism is hypercomplete, the restriction of $ \iuppershriek $ to hypercomplete sheaves defines a fully faithful right adjoint to $ \ilowerstar $, as desired.
\end{proof}

\begin{exm}\label{exm:profinitespacesembedintoPykS}
	Combining \Cref{wrn:PykS,cor:PykSashypercomp} with \Cref{cor:proembeds} shows that $ \Pyk(\Space) \equivalent \Space^{\dag} $ and provides a fully faithful embedding 
	\begin{equation*}
		\incto{\Pro(\Spacefin)}{\Pyk(\Space)} \period
	\end{equation*}
	In particular, the solidification $\XX^{\dag}$ of a bounded topos is naturally a $\Pyk$-algebra.
\end{exm}

\begin{exm}\label{exm:proetasPykalg}
	Combining \Cref{cor:proembeds} with \cite[Example 7.1.7]{Ultracategories} we see that that solidification of the étale topos $X_{\et}$ of a coherent scheme $X$ is the hypercompletion of the proétale topos $X_{\proet}$ of Bhatt and Scholze \cite{BhattScholzeProEtale}.
\end{exm}

\begin{wrn}
	In general, the solidification $\XX^{\dag}$ of a bounded coherent topos $\XX$ does \textit{not} coincide with its pyknotification $ \Pyk(\XX) $.
\end{wrn}




\section{Pyknotic categories}


\subsection{Pyknotic categories}

\begin{dfn}
	A \defn{pyknotic category} is a pyknotic object in $\Cat_{\delta}$ for some inaccessible cardinal $\delta$.
	A \defn{pyknotic functor} is a morphism of $\Pyk(\Cat_{\delta})$.

	Similarly, a \defn{pseudopyknotic category} is a pseudopyknotic object in $\Cat_{\delta}$ for some inaccessible cardinal $\delta$.
	A \defn{pseudopyknotic functor} is a morphism of $ \PsiPyk(\Cat_{\delta}) $.
\end{dfn}

\begin{nul}
	The inclusion $\Space_{\delta} \inclusion \Cat_{\delta}$ induces a fully faithful functor $\Pyk(\Space_{\delta}) \inclusion \Pyk(\Cat_{\delta})$.
	Write $ H \colon \fromto{\Cat_{\delta}}{\Space_{\delta}} $ for the left adjoint to the inclusion $ \incto{\Space_{\delta}}{\Cat_{\delta}} $, and $ \iota \colon \fromto{\Cat_{\delta}}{\Space_{\delta}} $ for its right adjoint.
	Then $ H(C) $ is the \textit{classifying space} obtained by inverting every morphism in $ C $, and $ \iota C \subset C $ is the \textit{interor} or maximal subgroupoid contained in $ C $.
	Since $ H $ and $ \iota $ both preserve finite products, composition with $ H $ and $ \iota $ define functors
	\begin{equation*}
		H,\iota \colon \fromto{\Pyk(\Cat_{\delta})}{\Pyk(\Space_{\delta})}
	\end{equation*}
	which are left and right adjoint to the inclusion $ \incto{\Pyk(\Space_{\delta})}{\Pyk(\Cat_{\delta})} $, respectively.
	These are the formations of the \defn{classifying pyknotic space} and the \defn{interior pyknotic space} of a pyknotic category.
\end{nul}

\begin{nul}
	The formation of the \defn{opposite (pseudo)pyknotic category} to a (pseudo)pyknotic category is performed objectwise.
\end{nul}

\begin{cnstr}
	If $ C $ is a $\Pyk$-module, then $ C $ acquires a natural pyknotic structure in the following manner.
	Let us abuse notation slightly and write $ C $ for the pyknotic category $\Stonean^{\op} \to \Pr^L$ given by
	\[
		C(K) \coloneq C \otimes_{\Pyk(\Space)} \Pyk(\Space)_{/K} \period
	\]
	The category underlying this pyknotic category $ C $ is indeed our original $ C $.
	Please observe also that if $K$ and $L$ are Stonean topological spaces, the natural morphism
	\[
		\begin{aligned}
    		C(K \sqcup L) &= C \otimes_{\Pyk(\Space)} \Pyk(\Space)_{/(K \sqcup L)} \\
    		&\simeq C \otimes_{\Pyk(\Space)} (\Pyk(\Space)_{/K} \oplus \Pyk(\Space)_{/L}) \\
    		&\to (C \otimes_{\Pyk(\Space)} \Pyk(\Space)_{/K}) \oplus (C \otimes_{\Pyk(\Space)} \Pyk(\Space)_{/L}) \\
    		&\simeq C(K) \times C(L)
  		\end{aligned}
	\]
	is an equivalence, so $ C $ is indeed a pyknotic category.

	In particular, if $f_{\ast} \colon \XX \to \Pyk(\Space)$ is a geometric morphism, then as a pyknotic category, $\XX$ carries a Stonean topological space $K$ to the fibre product of topoi
	\[
		\XX(K) \simeq \XX \times_{\Pyk(\Space)} \Pyk(\Space)_{/K} \period
	\]
\end{cnstr}

\begin{cnstr}\label{cnstr:twistedarrow}
	Let $ C $ be a pyknotic category.
	Composing $ C $ with the twisted arrow functor $\twarr \colon \Cat_{\delta} \to \Cat_{\delta}$ provides a \defn{twisted arrow pyknotic category} $\twarr(C)$ with its objectwise left fibration $\twarr(C) \to C^{\op} \times C$.
	Armed with this, we obtain a pyknotic mapping space functor
	\[
		\Map_C \colon C^{\op} \times C \to \Pyk(\Space)
	\]
	such that for any Stonean topological space $K$ and any pair of objects $X$ and $Y$ in $C(K)$, the sheaf $\Map_C(X,Y)(K)$ on $\Stonean_{/K}$ carries $f\colon L \to K$ to the space $\Map_{C(L)}(f^{\ast}X, f^{\ast}Y)$.
\end{cnstr}

\begin{exm}
	Let $\{C_{\alpha}\}_{\alpha\in\Lambda^{\op}}$ be an inverse system of categories.
	The limit
	\[
		C \coloneq \lim_{\alpha\in\Lambda^{\op}} C_{\alpha}^{\disc}
	\]
	of pyknotic categories is generally not discrete.
	The interior pyknotic space of $C$ is the limit of the discrete interiors $(\iota C_{\alpha})^{\disc}$, but the classifying pyknotic space $H(C)$ is not prodiscrete.
\end{exm}

\begin{exm}
	A \defn{stable pyknotic category} is a pyknotic object in the category $\Cat^{\stable}_{\delta}$ of ($\delta$-small) stable categories and exact functors.

	Since mapping spaces in pyknotic categories have natural pyknotic structures (\Cref{cnstr:twistedarrow}), it follows that the $\Ext$ groups in a stable pyknotic category admit the structure of pyknotic abelian groups.
	That is, if $A$ is a stable pyknotic category, then one may define, for any $n\in \ZZ$, the pyknotic abelian group
	\[
		\Ext^n_A(X, Y) \coloneq \pi_0\Map_A(X[-n], Y) \period
	\]

	The category $\Pyk(\Sp)$ of pyknotic spectra is naturally a $\Pyk$-algebra, and so for any module $A$ in $\Pr^L$ over $\Pyk(\Sp)$, the associated pyknotic category is a stable pyknotic category.
	In particular, for any pyknotic ring $R$, the pyknotic derived category $\DD_{\Pyk}(R)$ has the natural structure of a stable pyknotic category.
\end{exm}


\subsection{Pyknotic categories and complete Segal pyknotic spaces}

\begin{ntn}
	Let $ D $ be a category with finite limits.
	Write 
	\begin{equation*}
		\CS(D) \subset \Fun(\mbfDelta^{\op},D)
	\end{equation*}	
	for the full subcategory spanned by the \defn{complete Segal objects}, that is, those functors $ F \colon \fromto{\mbfDelta^{\op}}{D} $ satisfying the following conditions:
	\begin{itemize}
		\item For every $ m \in \mbfDelta $, the natural morphism
		\begin{equation*}
			\fromto{F_m}{F\{0,1\} \cross_{F\{1\}} F\{1,2\} \cross_{F\{2\}} \cdots \cross_{F\{m-1\}} F\{m-1,m\}}
		\end{equation*}
		is an equivalence in $ D $.

		\item The natural morphism
		\begin{equation*}
			\fromto{F_0}{F_3 \cross_{F\{0,2\} \cross F\{1,3\}} F_0}
		\end{equation*}
		is an equivalence in $ D $.
	\end{itemize}
\end{ntn}

\begin{nul}
	Joyal and Tierney \cite{JT} showed that the nerve construction defines an equivalence 
	\begin{equation*}
		N \colon \equivto{\Cat}{\CS(\Space)}
	\end{equation*}
	from the category of categories to the category of complete Segal spaces.
\end{nul}

From \Cref{lem:prodFuncommuteswithFun} we immediately deduce:

\begin{lem}\label{lem:pykCS}
	Let $ B $ be a category with products and $ D $ a category with finite limits.
	Then the natural equivalence of categories 
	\begin{equation*}
		\Funcross(B,\Fun(\mbfDelta^{\op},D)) \equivalent \Fun(\mbfDelta^{\op},\Fun^{\times}(B,D))
	\end{equation*}
	restricts to an equivalence
	\begin{equation*}
		\Funcross(B,\CS(D)) \equivalent \CS(\Funcross(B,D)) \period
	\end{equation*}
\end{lem}

\begin{exm}
	\Cref{lem:pykCS} provides an equivalence
	\begin{equation*}
		\Pyk(\Cat) \equivalent \CS(\Pyk(\Space)) \period
	\end{equation*}
\end{exm}

\begin{nul}
	Similarly, we have an equivalence
	\begin{equation*}
		\PsiPyk(\Cat) \equivalent \CS(\PsiPyk(\Space)) 
	\end{equation*}
	between pseudopyknotic categories and complete Segal objects in pseudopyknotic spaces.
\end{nul}


\subsection{Ultracategories as pseudopyknotic categories}

In recent work \cite{Ultracategories}, Lurie studied $ 1 $-categories equipped with an \textit{ultrastructure}, which we simply refer to as \textit{$ 1 $-ultracategories}\footnote{Here we still follow our categorical conventions and use the term `$ 1 $-ultracategory' to refer to what Lurie calls an `ultracategory' in \cite{Ultracategories}, and use the term `ultracategory' for the higher-categorical notion.}.
An ultracategory structure on a $ 1 $-category $ M $ consists of, for each set $ S $ and ultrafilter $ \mu \in \beta(S) $, an \textit{ultraproduct} functor
\begin{equation*}
	\int_{S} (-)d\mu \colon \fromto{\prod_{s \in S} M}{M} \comma
\end{equation*}
along with data relating these ultraproduct functors suggested by the integral notation, all subject to a number of coherence axioms \cite[Definition 1.3.1]{Ultracategories}. 
The primary example of a $ 1 $-ultracategory is the following:

\begin{exm}[{\cite[Example 1.3.8]{Ultracategories}}]
	Let $ M $ be a $ 1 $-category with products and filtered colimits.
	Then $ M $ has an ultrastructure where for each set $ S $ and ultrafilter $ \mu \in \beta(S) $, the ultraproduct $ \int_S (-) d\mu $ is defined by the usual ultraproduct formula
	\begin{equation}\label{eq:catultraprod}
		\int_{S} m_s d\mu \coloneq \colim_{S_0 \in \mu} \prod_{s \in S_0} m_s \comma
	\end{equation}
	where the colimit is taken over the filtered diagram of all subsets $ S_0 \subset S $ in the ultrafilter $ \mu $.
	This ultrastructure is called the \textit{categorical ultrastructure} on $ M $.

	More generally, if $ M' \subset M $ is a full subcategory closed under ultraproducts in $ M $ (defined by \cref{eq:catultraprod}), then the categorical ultrastructure on $ M $ restructs to an ultrastructure on $ M' $.
	In fact, every $ 1 $-ultracategory can be obtained in this way; see \cite[\S8]{Ultracategories}.
\end{exm}

\begin{rec}
	Let $ \XX $ be a $ 1 $-topos.
	The \defn{category of points} of $ \XX $ is the category $ \Pt(\XX) \coloneq \Fun^{\ast}(\XX,\Set) $ of left exact left adjoints $ \fupperstar \colon \fromto{\XX}{\Set} $ and natural transformations between them.
	If $ \XX $ is a coherent $ 1 $-topos, then restriction along the inclusion $ \incto{\XXcoh}{\XX} $ of coherent objects defines a fully faithful functor
	\begin{equation*}
		\incto{\Pt(\XX)}{\Fun(\XXcoh,\Set)}
	\end{equation*}
	with essential image the \defn{pretopos morphisms}, i.e., those functors $ \fromto{\XXcoh}{\Set} $ that preserve finite limits, finite coproducts, and effective epimorphisms.
\end{rec}

\begin{exm}\label{exm:ultrastronPtX}
	If $ \XX $ is a coherent $ 1 $-topos, then by the Łoś ultraproduct theorem \cite[Theorem 2.1.1]{Ultracategories} and the equivalence between coherent $ 1 $-topoi and $ 1 $-pretopoi \cite[Proposition C.6.4]{Ultracategories}, the cateory of points $ \Pt(\XX) $ is closed under ultraproducts in $ \Fun(\XX^{\coh},\Set) $, hence admits an ultrastructure.
\end{exm}

If $ M $ and $ N $ are $ 1 $-ultracategories, a \defn{left ultrastrucature} on a functor $ F \colon \fromto{M}{N} $ consists of comparison natural transformation of ultraproducts
\begin{equation}\label{eq:leftultcomparison}
	\fromto{F\paren{\textstyle\int_{S} (-)d\mu}}{\textstyle\int_{S} F(-) d\mu} 
\end{equation} 
for each set $ S $ and ultrafilter $ \mu \in S $, subject to a number of coherences \cite[Definition 1.4.1]{Ultracategories}.
A left ultrafunctor is an \defn{ultrafunctor} if all of the comparison morphisms \eqref{eq:leftultcomparison} are equivalences.
Then $ 1 $-ultracategories and left ultrafunctors between them assemble into a $ 2 $-category $ \UltL_1 $.
The $ 2 $-category $ \UltL_1 $ embeds into pseudopyknotic $ 1 $-categories in the following manner.
First, writing $ \USet \subset \UltL_1 $ for the full subcategory spanned by those $ 1 $-ultracategories whose underlying $ 1 $-category is discrete, there is an equivalence of $ 1 $-categories
\begin{equation*}
	\equivto{\USet}{\Comp}
\end{equation*}
\cite[Theorem 3.1.5]{Ultracategories}.
Moreover, in \cite[\S4]{Ultracategories} Lurie proves that for any $ 1 $-ultracategory $ M $, the functor
\begin{equation*}
	\Funlult(-,M) \colon \fromto{\Comp^{\op} \equivalent \USet^{\op}}{\Cat_{1}}
\end{equation*}
given by sending a compactum $ K $ to the $ 1 $-category $ \Funlult(K,M) $ of left ultrafunctors $ \fromto{K}{M} $ defines a stack of $ 1 $-categories on $ \Comp $ with respect to the effective epimorphism topology.
Moreover, the construction
\begin{equation*}
	\fromto{\UltL_1}{\PsiPyk(\Cat_{1})} \comma \qquad \goesto{M}{\Funlult(-,M)}
\end{equation*}
defines a fully faithful embedding.\footnote{Ultracategories and left ultrafunctors really form a $ (2,2) $-category, and ultracategories and left ultrafunctors embed fully faithfully into the $ (2,2) $-category of $ 1 $-categories, functors, and natural transformations.}

The main motivation of the study of $ 1 $-ultracategories is the following result, which implies both the Deligne Completeness Theorem and Makkai's Strong Conecptual Completeness Theorem \cite{MR900266}:

\begin{thm}[{\cite[Theorem 2.2.2]{Ultracategories}}]
	Let $ \XX $ be a coherent $ 1 $-topos.
	Then there is a natural equivalence of categories
	\begin{equation*}
		\equivto{\Funlult(\Pt(\XX),\Set)}{\XX} \comma
	\end{equation*}	
	where $ \Set $ is given the categorical ultrastructure and $ \Pt(\XX) $ is given the ultrastructure of \Cref{exm:ultrastronPtX}.
\end{thm}


The `explicit' definition of an $ 1 $-ultracategory as a $ 1 $-category with ultraproduct functors subject to a collection of coherences isn't well-suited to generalise to the higher-categorical setting.
As such, we provide a different description of ultracategories following \cite[\S8]{Ultracategories}; this material will appear in \cite{Kerodon}, so we do not provide proofs here.

\begin{dfn}
	Let $ E $ be a category with finite products.
	An object $ X \in E $ with finite products is \defn{coconnected} if $ \Map_{E}(-,X) \colon \fromto{E^{\op}}{\Space} $ carries finite products in $ E $ to finite coproducts in $ \Space $.
\end{dfn}

\begin{dfn}[{\cite[Definition 8.2.2]{Ultracategories}}]\label{def:ultracatenvelope}
	An \defn{ultracategory envelope} is a category $ E $ satisfying the following conditions:
	\begin{enumerate}[(\ref*{def:ultracatenvelope}.1)]
		\item The category $ E $ admits products.

		\item Every object $ X \in E $ can be written as a product $ \prod_{s \in S} X_s $, where each factor $ X_s $ is a coconnected object of $ E $.

		\item The full subcategory $ E^{\cc} \subset E $ spanned by the coconnected objects has ultraproducts in $ E $.
		That is, for every collection $ \{X_s\}_{s \in S} $ of coconnected objects of $ E $ and every ultrafilter $ \mu $ on $ S $, the filtered colimit
		\begin{equation*}
			\colim_{S_0 \in \mu} \prod_{s \in S_0} m_s \comma
		\end{equation*}
		exists and is a coconnected object of $ E $.
	\end{enumerate}
\end{dfn}

\begin{dfn}\label{def:ultracatsviaenv}
	Let $ M $ be a category.
	An \defn{ultracategory structure} on $ M $ consists of an ultracategory envelope $ \Env(M) $ along with an equivalence of categories $ \equivto{M}{\Env(M)^{\cc}} $.
\end{dfn}

\begin{nul}
	Lurie shows \cite[Theorem 8.2.5]{Ultracategories} that the theory of $ 1 $-ultracategories in the sense of \Cref{def:ultracatsviaenv} coincides with the `explicit' theory of ultracategories (in the sense of \cite[Definition 1.3.1]{Ultracategories}).
\end{nul}

\begin{exm}[{\cite[Example 8.4.3]{Ultracategories}}]
	Let $ \XX $ be a bounded coherent topos.
	Then again, restriction along the inclusion $ \incto{\XXcohbdd}{\XX} $ defines an embedding
	\begin{equation*}
		\incto{\Pt(\XX)}{\Fun(\XXcohbdd,\Space)} 
	\end{equation*}
	\SAG{Proposition}{A.6.4.4}.
	Write $ \Env(\Pt(\XX)) \subset \Fun(\XXcohbdd,\Space) $ for the smallest full subcategory containing $ \Pt(\XX) $ and closed under small products.
	Then $ \Env(\Pt(\XX)) $ is an ultracategory envelope and the inclusion $ \Pt(\XX) \subset \Env(\Pt(\XX)) $ provides an ultrastructure on $ \Pt(\XX) $.
\end{exm}

More natural from the ultracategory envelope perspective are \defn{right ultrafunctors} -- morphisms of the ultracategory envelopes that preserve products and coconnected objects \cite[\S8.2]{Ultracategories}.
In terms of the explicit definition of $ 1 $-ultracategories, right ultrafunctors are just like left ultrafunctors, but the ultraproduct comparisons \eqref{eq:leftultcomparison} go in the opposite direction \cite[Definition 8.1.1]{Ultracategories}.
From the ultracategory envelope perspective, defining left ultrafunctors is more involved, but the upshot is that there's still a fully faithful embedding into pseudopyknotic categories:

\begin{thm}
	There is a fully faithful embedding
	\begin{equation*}
		\incto{\UltL}{\PsiPyk(\Cat)}
	\end{equation*}
	from a category of ultracategories and left ultrafunctors between them to pseudopyknotic categories.
\end{thm}

\begin{exm}\label{exm:exodromyaspyknotic}
	The assignment $ \goesto{\XX}{\Pt(\XX)} $ defines a fully faithful functor from bounded coherent topoi and \textit{arbitrary} geometric morphisms to ultracategories and left ultrafunctors -- \textit{coherent} geometric morphisms are identified \textit{ultrafunctors} (cf. \cite[Example 2.2.8]{Ultracategories}).
	
	Consider the category $ \Strat_{\pi}^{\natural} $ of \defn{$ \pi $-finite stratified spaces}\footnote{Here we work with the natural stratification by the underlying poset.}.
	This is the full subcategory $ \Strat_{\pi}^{\natural} \subset \Cat $ with objects those categories $ \Pi $ with the property that every endomorphism in $ \Pi $ is an equivalence, $ \Pi $ has only finitely many objects up to equivalence, and all of the mapping spaces in $ \Pi $ are $ \pi $-finite spaces.
	In \cite{exodromy} showed that the extension to proöbjects of the functor given by $ \goesto{\Pi}{\Fun(\Pi,\Space)} $ defines a fully faithful embedding
	\begin{equation*}
		\incto{\Pro(\Strat_{\pi}^{\natural})}{\Topbc}
	\end{equation*}
	of profinite stratified spaces into bounded coherent topoi and coherent geometric morphisms.
	We identified the essential image as the category $ \Top_{\infty}^{\spec} $ of \textit{spectral} topoi -- this is our higher-categorical Hoschster Duality Theorem \cite[Theorem 10.3.1]{exodromy}.

	This embedding has a left adjoint $ \Pi_{(\infty,1)}^{\wedge} \colon \fromto{\Topbc}{\Pro(\Strat_{\pi}^{\natural})} $ given by the \textit{profinite stratified shape}.
	For a spectral topos $ \XX $, the profinite statified shape $ \Pi_{(\infty,1)}^{\wedge}(\XX) $ has the property that the materialisation $ \mat \Pi_{(\infty,1)}^{\wedge}(\XX) $ is equivalent to the category $ \Pt(\XX) $ of points of $ \XX $.
	It is thus possible to recast the profinite stratified shape and \textit{exodromy equivalence} of \cite[Theorem 11.1.7]{exodromy} in terms of ultracategories (or pseudopyknotic categories).
	In particular, for a coherent scheme $ X $, our profinite Galois category $ \Gal(X) $\cites{Barwick:galperf}[\S13]{exodromy} is naturally a pyknotic category.
	The benefit of the perspective taken in \cite{exodromy} is that the theory of profinite stratified spaces is appreciably more simple than that of pyknotic categories.
\end{exm}


\DeclareFieldFormat{labelnumberwidth}{#1}
\printbibliography[keyword=alph]
\addcontentsline{toc}{section}{References} 
\DeclareFieldFormat{labelnumberwidth}{{#1\adddot\midsentence}}
\printbibliography[heading=none, notkeyword=alph]

\end{document}